\documentclass[reqno]{amsart}
\usepackage[utf8]{inputenc}
\usepackage{amsthm}
\usepackage{amsmath}
\usepackage{amssymb}
\usepackage{ wasysym }
\usepackage{mathrsfs}
\usepackage[all,cmtip]{xy}
\usepackage{enumerate}
\usepackage{hyperref}
\usepackage{ragged2e} 
\usepackage{color}
\usepackage{verbatim}
\usepackage[shortlabels]{enumitem}
\usepackage[normalem]{ulem}
\usepackage{cancel}

\theoremstyle{remark}

\theoremstyle{plain}
\newtheorem{defin}{Definition}[section]
\newtheorem{definp}[defin]{Definition-Proposition}
\newtheorem{prop}[defin]{Proposition}
\newtheorem{thm}[defin]{Theorem}
\newtheorem{cor}[defin]{Corollary}
\newtheorem{lm}[defin]{Lemma}
\theoremstyle{remark}
\newtheorem{rmk}[defin]{Remark}

\theoremstyle{remark}
\newcommand{\prodscal}[2]{\left\langle#1,#2\right\rangle}
\usepackage[utf8]{inputenc}

\newcommand{\CC}{\mathbb{C}}

\newcommand{\GG}{\mathbb{G}}
\newcommand{\HH}{\mathbb{H}}

\newcommand{\RR}{\mathbb{R}}

\newcommand{\ZZ}{\mathbb{Z}}

\newcommand{\sA}{\mathcal{A}}
\newcommand{\sB}{\mathcal{B}}

\newcommand{\sD}{\mathcal{D}}

\newcommand{\sF}{\mathcal{F}}

\newcommand{\sL}{\mathcal{L}}
\newcommand{\sM}{\mathcal{M}}

\newcommand{\sV}{\mathcal{V}}
\newcommand{\sW}{\mathcal{W}}
\newcommand{\sX}{\mathcal{X}}

\DeclareMathOperator{\Hom}{Hom}
\DeclareMathOperator{\End}{End}

\newcommand{\btimes}{\hat{\otimes}}
\newcommand{\id}{\iota}
\newcommand{\flip}{\mathrm{Flip}}
\newcommand{\mult}{\mathrm{m}}
\newcommand{\op}{\mathrm{op}}
\newcommand{\cop}{\mathrm{cop}}
\newcommand{\onto}{\twoheadrightarrow}
\newcommand{\Dom}{\mathrm{Dom}}
\newcommand{\sDom}{\mathcal{D}\!\mathit{om}}

\begin{document}

\title{Bornological quantum groups as locally compact quantum groups}
\author{Damien Rivet}
\address{Université Clermont Auvergne, CNRS, LMBP, F-63000 Clermont-Ferrand, France}
\email{damien.rivet@uca.fr}
\author{Robert Yuncken}
\address{Université Clermont Auvergne, CNRS, LMBP, F-63000 Clermont-Ferrand, France}
\email{robert.yuncken@uca.fr}

\subjclass[2010]{20G42, 16T05, 46L65, 46L51}

\thanks{
Both authors were supported by the CNRS PICS project OpPsi and by the project SINGSTAR of the Agence Nationale de la Recherche, ANR-14-CE25-0012-01.}

\keywords{Quantum groups, bornological algebras, algebraic quantum groups, locally compact quantum groups, operator algebras}

\maketitle

\begin{abstract}
    Bornological quantum groups were introduced by Voigt in order to generalize the theory of algebraic quantum groups in the sense of van Daele. In particular the class of bornological quantum groups contains all classical locally compact groups. In this paper we prove that a bornological quantum group gives rise to a locally compact quantum group, in a similar way to Kustermans and van Daele's result for algebraic quantum groups. We show that the bornological quantum groups, although more general than the algebraic ones, share most of their nice properties. We also argue that bornological quantum groups, when they occur as dense subalgebras of locally compact quantum groups, are useful tools for studying locally compact quantum groups.  For instance, we give a very simple definition of a closed quantum subgroup in the bornological framework, and show that it yields a closed subgroup of the corresponding locally compact quantum groups in the sense of Vaes or Wooronowicz.  The results on quantum subgroups are new even in the context of algebraic quantum groups.
    \end{abstract}
\section{Introduction}

The notion of locally compact quantum group was successfully axiomatized by Kustermans and Vaes \cite{KV:LCQG}. It gives a very general class of quantum groups closed under Pontryagin duality. It is powerful setting to study a large class of examples of quantum groups and their representation theory.  In this framework we have access to the $C^*$-algebra of continuous functions of our quantum group and the Von Neumann algebra of $L^{\infty}$ functions, as well as the group $C^*$ and von Neumann algebras. However, in the classical case if one has a locally compact group $G$, the algebra $C_c(G)$ (or $C_c^{\infty}(G)$, or $\mathcal{S}(G)$ the class of Schwartz functions on $G$ if $G$ is a Lie group) is more suitable for many explicit constructions in representation theory. Unfortunately in the quantum case, it is still unknown if one can extract such algebras from the axioms of Kustermans and Vaes.

The algebras  $C_c(G)$, $C_c^{\infty}(G)$, $\mathcal{S}(G)$ could of course be seen as topological algebras, but for many reasons it is more convenient to consider them as bornological algebras (see for example \cite{MEY}).
In \cite{Born}, C. Voigt developed an axiomatization for quantum analogs of such algebras.  This generalizes the notion of algebraic quantum group introduced by A. Van Daele in \cite{Vandaele}, although Voigt does not include a $*$-structure. In this paper we show, among other things, that if one adds a $*$-structure to Voigt's axioms then the resulting bornological quantum groups give rise to locally compact quantum groups.

Many known examples of quantum groups can be seen as algebraic quantum groups (for example, complex semisimple quantum groups (\cite{yuncken-voigt}) but the class of algebraic quantum group does not include classical non compact Lie groups. 
The bornological framework is a setting where most of the results we know for algebraic quantum groups remain valid and which is almost as easy to handle (see \cite{riv} for an example of application in representation theory) and yet which is large enough to contain the classical groups.

The main motivation of this paper is to make clear the link between bornological and $C^*$-algebraic quantum groups. Indeed when one starts with a bornological quantum group, it can be very effective to use the tools of the Kustermans and Vaes framework, and conversely when one studies a $C^*$-algebraic quantum group, the access to a bornological dense subalgebra allows us to make simpler explicit computations.

To this end, we prove the following theorem, which is an analogue of a theorem of Kustermans and van Daele \cite{KvD} in the algebraic setting.

\begin{thm}
 Let $\sA(\GG)$ be the algebra of functions on a bornological quantum group $\GG$ (see Section \ref{sec:bqg} for the definitions), with left Haar-integral $\phi_\GG$.  Then the regular representation of $\sA(G)$ on the GNS Hilbert space $L^2(\GG)$ with respect to the left Haar integral is a $*$-representation by bounded operators and the norm closure of $\sA(\GG)$ in $\sB(L^2(\GG))$ is naturally a $C^*$-algebraic locally compact quantum group in the sense of Kustermans-Vaes \cite{KV:LCQG}.
\end{thm}
 
As those familiar with the work of Van Daele and his collaborators will know, many of the intermediate results required to prove this theorem are important in their own right.  For instance, we will prove that the complex powers of the modular element exist in the bornological framework.

\begin{thm}
  Let $\delta_\GG$ be the modular element of a bornological quantum group $\GG$.  Then $\delta_\GG$ extends to an unbounded self-adjoint element on $L^2(\GG)$ whose complex powers restrict back to a one-parameter family $(\delta_\GG^z)_{z\in\GG}$ of group-like elements in the bornological multiplier algebra $\sM(\sA(\GG))$.
\end{thm}

Similarly, we prove that the modular automorphism group $(\sigma_z)_{z\in\CC}$ and the scaling group $(\tau_z)_{z\in\CC}$ restrict to automorphisms of the bornological algebra $\sA(\GG)$.  In fact, we will show that all of these automorphism groups can be understood in terms of the complex powers of the modular element $\delta_\GG$ and of the modular element $\delta_{\hat\GG}$ for the Pontryagin dual.  In particular, the complex powers of the Hopf automorphism $S^2$ are obtained from a generalization of Radford's Formula, which we state imprecisely as follows.

\begin{thm}
  The bornological Hopf automorphism $S^2$ fits into a complex $1$-parameter family of Hopf automorphisms of $\sA(\GG)$ given by
 $$
 (S^2)^{z}(f)=\delta_{\mathbb{G}}^{-iz/2}(\delta_{\hat{\mathbb{G}}}^{iz/2}*f*\delta_{\hat{\mathbb{G}}}^{-iz/2})\delta_{\mathbb{G}}^{iz/2}. $$  
\end{thm}

For a precise statement, see Theorem \ref{thm:S2z}.

Finally, as an important illustration of the advantages of the bornological framework, we will discuss the notion of closed quantum subgroups.  Here, a closed quantum subgroup can be defined very simply as the image of a surjective morphism.  This is in contrast to the locally compact framework, where there are at least two potential definitions of a closed quantum subgroup, due to Vaes \cite{Vaes} and Woronowicz, both of which are reasonably intuitive to an expert in operator algebras but which are certainly more cumbersome to work with.  For an excellent discussion of this situation, see \cite{Daws}.  We show that the naïve definition of a closed quantum subgroup in the bornological framework yields a closed quantum subgroup of the corresponding locally compact quantum group in the sense of both Vaes and Woronowicz.  Thus our results provide a new and potentially simpler means of exhibiting closed subgroups of locally compact quantum groups, by passing via a dense algebraic or bornological subalgebra.

\medskip

The paper is organized as follows. In Section \ref{sec:bqg} we recall important results on bornological quantum groups from \cite{Born}, and we modify Voigt's original definition to add a $*$-structure. 
In Section \ref{sec:lcqg}, we follow the approach of Kustermans and Van Daele \cite{KvD} to build the reduced $C^*$-algebraic quantum group associated to a bornological quantum group. 
Many of the results have close analogues in the algebraic context, in which case we remain brief, often referring to existing proofs in \cite{KvD}.  However, our approach to the modular group is somewhat different to that in \cite{KvD}, making use of the complex powers of the modular element and of the automorphism $S^2$ to streamline the calculations, and here we give a fuller treatment.  
Finally, in Section \ref{sec:subgroups} we show that the natural definition of a closed quantum subgroup in the bornological setting gives rise to closed quantum subgroups in the senses of Vaes \cite{Vaes} and of Woronowicz \cite{Daws}.

We finish with a technical remark.  We have written this article under the global assumption that the scaling constant of our bornological quantum group is $1$.  This is the case for all of the examples we are interested in, and simplifies the exposition. In the case of algebraic quantum groups, it is known that the scaling constant is always $1$, see \cite{DeCVan}.  We do not know whether this is also the case for bornological quantum groups.

\section{Bornological quantum groups}
\label{sec:bqg}

The purpose of this section is to recall the basic results about bornological quantum groups from \cite{Born}, while also adding remarks about the addition of an involution to the axioms.  Most of these results are generalizations of results on algebraic quantum groups \cite{VanDaele:multiplier_Hopf_algebras, Vandaele, KvD}.

\subsection{Definitions}

All vector spaces in this paper will be over $\CC$.  

A \emph{bornological vector space} is a vector space $\sV$ equipped with a family $\sB$ of subsets satisfying the following three axioms:
\begin{enumerate}
	\item $\bigcup_{b\in\sB} B = \sV$,
	\item if $B\in\sB$ and $A\subset B$ then $A\in\sB$,
	\item if $B_1,B_2\in\sB$ then $B_1\cup B_2 \in \sB$,
\end{enumerate}
and which is stable under addition and scalar multiplication.  The set $\sB$ is called the \emph{bornology} of $\sV$ and elements of $\sB$ are called \emph{bounded sets}.  

A topological vector space $\sV$ can be equipped with two standard bornologies.  Firstly, there is the \emph{von Neumann bornology} in which a subset $B$ is bounded if and only if every continuous semi-norm on $\sV$ is bounded on $B$.  Secondly, there is the \emph{precompact bornology} in which $B$ is bounded if and only if it is precompact, meaning that for any open neighbourhood $U$ of $0$ there is a finite set $F\subset\sV$ such that $B\subseteq F+U$.  For many applications, the precompact bornology is the better behaved.

A linear map between bornological vector spaces is \emph{bounded} if it sends bounded sets to bounded sets.

Some standard assumptions of convexity and completeness will be required.  
We recall that a subset $X$ of a complex vector space $\sV$ is \emph{balanced} if $\alpha X \subseteq X$ for every $\alpha\in\CC$ with $|\alpha|\leq 1$.  Balanced convex subsets will be referred to as \emph{disks}.  Any disk $D$ gives rise to a seminorm on its span, $p_D(v) = \inf\{r>0\mid v\in rD\}$, called the \emph{gauge} with respect to $D$.  If the span of $D$ is a Banach space with respect to the gauge, then $D$ is called a \emph{completant disk}.

A bornological vector space is \emph{convex} if the bornology is closed under taking balanced convex hulls, and \emph{complete} if every bounded set is contained in a completant disk.  In this article, all bornological vector spaces will be assumed to be complete and convex unless explicitly stated otherwise.
In fact, the structure of complete convex bornological space is equivalent to realizing $\sV$ as a direct limit of an injective system of Banach spaces, see \cite[Theorem A.4]{Meyer:thesis}.  

There is a natural, and essentially uniquely defined, tensor product in the category of complete, convex bornological vector spaces, which we denote by $\hat{\otimes}$.  It behaves well with respect to the Hom-tensor adjunction:
\[
 \Hom(\sV\btimes\sW,\sX) \cong \Hom(\sV,\Hom(\sW,\sX)).
\]
This makes the bornological category extremely well-adapted for the study of quantum group algebras, where tensor products are ubiquitous.


We write $\sV^* = \Hom(\sV,\CC)$ for the space of bounded linear functionals on $\sV$.

As in \cite{Meyer:born_vs_top, Born}, we will also impose that our bornological vector spaces satisfy the \emph{approximation property}, meaning that the identity map on $\sV$ can be approximated uniformly on compact sets by finite rank operators.  This is the case for all the bornological vector spaces which we shall encounter.

A \emph{bornological algebra} is a bornological vector space $\sA$ equipped with a bounded algebra product $\sA\times\sA\to\sA$.  It therefore extends to a bounded map $\mult:\sA\btimes\sA\to\sA$.  
Similarly, a bornological module $\sV$ over $\sA$ can be defined in terms of a bounded action $\sA\btimes\sV\to\sV$.  

A bornological module is \emph{essential} if every bounded set of $\sV$ is the image of a bounded set in $\sA\btimes\sV$.  A bornological algebra is \emph{essential} if it is essential as a module over itself.  A bounded morphism of bornological algberas $\phi:\sA\to\sB$ is \emph{essential} if it makes $\sB$ into an essential $\sA$-module.

A \emph{multiplier} of a bornological algebra is a pair of bounded maps $c=(c_l\cdot,\cdot c_r)$ from $\sA$ to itself satisfying
\[
 c_l\cdot(ab) = (c_l\cdot a)b, \quad 
 (ab)\cdot c_r = a(b\cdot c_r), \quad
 (a\cdot c_r)b = a(c_l\cdot b),
\]
for all $a,b\in\sA$.
The multipliers form a bornological algebra $\sM(\sA)$, with the bornology restricted from $\End(\sA)\oplus\End(\sA)$, and $\sA$ sits in $\sM(\sA)$ as an ideal.  We may thus suppress the dots and the subscripts $l$ and $r$ in the notation.  For details, see \cite{Born}.  

It is an important fact that if $\sA$ and $\sB$ are essential bornological algebras then any essential morphism  $\phi:\sA\to\sM(B)$ extends uniquely to a morphism on the multipliers $\phi:\sM(\sA) \to \sM(\sB)$.  We shall use this frequently without mention.

A bornological algebra equipped with a bounded antilinear involution is called a \emph{bornological $*$-algebra}.  
If $\sA$ is a bornological $*$-algebra then so is $\sM(\sA)$.

We will make an unconventional choice of notation here.  As we will describe later, the space of functions on a bornological quantum group is equipped with two distinct $*$-structures, associated to the pointwise product and the convolution product, respectively, and it will be important to distinguish them.  We will therefore use the notation $a\mapsto \overline{a}$ for the ``pointwise adjoint'' and $a\mapsto a^*$ for the ``convolution adjoint''.  In particular, the reader should keep in mind that $\overline{ab} = \overline{b}\,\overline{a}$.  

We equip the tensor product $\sA\btimes\sB$ of two bornological $*$-algebras with the involution defined by
\[
 \overline{(a\otimes b)} = \overline{a} \otimes \overline{b}.
\]

A \emph{coproduct} on a bornological $*$-algebra $\sA$ is an essential bounded $*$-homomorphism $\Delta : \sA \to \sM(\sA\btimes\sA)$ which is coassociative, meaning $(\id\hat\otimes\Delta)\Delta = (\Delta\hat\otimes\id)\Delta$ as maps from $\sA$ to $\sM(\sA\btimes\sA\btimes\sA)$, and such that the \emph{Galois maps}
\begin{align*}
  &\gamma_l 
  :a\otimes b \mapsto (\Delta a) (b\otimes 1), &
  &\rho_r 
  :a\otimes b \mapsto (1\otimes a) (\Delta b)
\end{align*}
are bounded linear maps from $\sA\btimes\sA$ to itself.
We write $\Delta^\cop = \flip \circ \Delta$ for the co-opposite comultiplication, where $\flip$ denotes the flip map.
We note that Voigt \cite{Born} does not impose the condition on the Galois maps in his definition of a bornological coproduct, although he does require it as a hypothesis in all his successive results.  

We also define the maps 
\begin{align*}
  &\gamma_r 
  :a\otimes b \mapsto (\Delta a) (1\otimes b),  &
  &\rho_l 
  :a\otimes b \mapsto (a\otimes 1) (\Delta b),
\end{align*}	
as well as the variants $\gamma_l^\op$, $\gamma_l^\cop$, $\gamma_l^{\op,\cop}$, \emph{etc}, in which we replace the multiplication by $\mult^\op$ and/or the comultiplication by $\Delta^\cop$.  The resulting sixteen maps from $\sA\btimes\sA \to \sA\btimes\sA$ will all be referred to as \emph{Galois maps}.  They all map $\sA\btimes\sA$ into itself because they can all be related to $\gamma_l$ and $\rho_r$ via the flip maps and conjugation by the involution.

This condition on the Galois maps allows us to define, for any $a\in\sA$ and $\omega\in\sA^*$, a multiplier
$
 (\id\hat\otimes\omega)(\Delta(a)) \in \sM(\sA)
$
by
\begin{align*}
 b\cdot(\id\hat\otimes\omega)(\Delta(a)) &= (\id\hat\otimes\omega)((b\otimes1)\Delta(a)), \\
 (\id\hat\otimes\omega)(\Delta(a))\cdot b &= (\id\hat\otimes\omega)(\Delta(a)(b\otimes1)), 
\end{align*}
where $b\in\sA$.  We can define the multiplier 
$
 (\omega\hat\otimes\id)(\Delta(a)) \in \sM(\sA)
$
similarly.

Some notational remarks are in order.  Firstly, if $\omega \in \sA^*$ and $b,c\in\sA$, we will use the notation $b\omega c$ for the linear functional $a\mapsto \omega(cab)$.  This notation will be generalized to linear functionals on other algebras.

Secondly, to simplify formulas, we will often use Sweedler notation for the coproduct, writing
\begin{align*}
 \Delta (a) &= a_{(1)} \otimes a_{(2)},
\end{align*}
where $a\in\sA$.  For classical Hopf algebras, this can be understood as a summation convention, but here it is a purely formal notation.  
That is, the terms $a_{(1)}$ and $a_{(2)}$ have no meaning on their own, but are only placeholders for the position of a coproduct in the legs of the multipliers of the bornological tensor product $\sM(\sA\btimes\sA)$.
Thus, for instance, we can write the Galois maps as
\[
 \gamma_l(a\otimes b) = a_{(1)}b \otimes a_{(2)}, \quad 
 \gamma_l^\cop(a\otimes b) = a_{(2)}b \otimes a_{(1)}, \quad
 \text{\emph{etc.}}
\]
We extend this to iterated coproducts in the usual way, writing
\[
  \Delta^{(n)}(a) = a_{(1)} \otimes \cdots \otimes a_{(n+1)}.
\]

Thanks to the fact that all Galois maps have image in $\sA\btimes\sA$, given any elements $a, b_1,\ldots,b_{n-1} \in \sA$ the product $b_1a_{(1)} \otimes b_2a_{(2)} \otimes \cdots \otimes a_{(i)} \otimes\cdots \otimes b_{n-1}a_{(n)}$ belongs to $\sA^{\btimes n}$, where exactly one of the legs $a_{(i)}$ is not multiplied by an element $b_i$ of $\sA$.  The same is true if any number of the $b_i$ is multiplied on the right instead of the left.

A coproduct $\Delta$ on a bornological $*$-algebra $\sA$ is said to satisfy the \emph{cancellation property} if the Galois maps $\gamma_l$ and $\rho_r$ are linear bornological isomorphisms from $\sA\btimes\sA$ to itself.  Once again, this implies all sixteen Galois maps are linear bornological isomorphisms from $\sA\btimes\sA$ to itself. 

A bounded linear functional $\phi$ on a bornological $*$-algebra is \emph{positive} if $\phi(\overline{a}\,a)\geq 0$ for all $a\in\sA$.  This implies $\phi(\overline{a})= \overline{\phi(a)}$ for all $a\in\sA$. 

A \emph{left-invariant integral} on an essential bornological $*$-algebra $\sA$ with coproduct is a bounded linear functional $\phi\in\sA^*$ such that 
\[
 (\id\hat\otimes\phi)(\Delta(a)) = \phi(a)1
\]
for all $a\in\sA$.  Similarly, a \emph{right-invariant integral}
is $\psi\in\sA^*$ such that
\[
 (\psi\hat\otimes\id)(\Delta(a)) = \psi(a)1.
\]

The following theorem is due to Voigt \cite{Born}, see also Van Daele \cite{VanDaele:multiplier_Hopf_algebras}.

\begin{thm}
\label{thm:BQG-definition}
 Let $\sA$ be a bornological $*$-algebra equipped with a coproduct $\Delta$ and a positive faithful left invariant integral $\phi$.  The following are equivalent:
 \begin{enumerate}[(i)]
  \item \label{item:BQG2}
   $\Delta$ satisfies the cancellation property,
   
  \item \label{item:BQG1}
   there exists an essential homomorphism $\epsilon:\sA\to\CC$, called the {counit}, and a bounded algebra antiautomorphism coalgebra antiautomorphism $S:\sA\to\sA$, called the {antipode}, satisfying the following Hopf-type axioms: For all $a,b\in\sA(G)$,
   \[
    (\epsilon\hat\otimes\id)(\Delta(a)) = a = (\id\hat\otimes\epsilon)(\Delta(a))
		\]
   and
   \[
    \mult(S\hat\otimes\id)(\Delta(a)(1\otimes b)) = \epsilon(a)b, \qquad
    \mult(\id\hat\otimes S)((a\otimes 1)\Delta(b)) = \epsilon(b)a .
   \]
 \end{enumerate}
 In this case, the maps $\epsilon$ and $S$ are uniquely defined and satisfy
 \[
  \epsilon(\overline{a}) = \overline{\epsilon(a)}, \qquad
  S(\overline{a}) = \overline{S^{-1}(a)}.
 \]
\end{thm}

\begin{proof}
 The only new point here is the compatibility with the involution.  Define the map $\overline{\epsilon}:\sA \to \CC$ by $\overline{\epsilon}(a) = \overline{\epsilon(\overline{a})}$.  Then
 \[
  (\overline{\epsilon}\hat\otimes\id)(\Delta(a))
   = \overline{(\epsilon\hat\otimes\id)(\Delta(\overline{a}))} = a = \overline{(\id\hat\otimes\epsilon)(\Delta(\overline{a}))} = (\id\hat\otimes\overline{\epsilon})(\Delta(a)).
 \]
 By uniqueness of the counit we have $\overline{\epsilon}=\epsilon$.  Similarly, if we define $\overline{S}:\sA\to\sA$ by $\overline{S}(a) = \overline{S^{-1}(\overline{a})}$ then $\overline{S}$ satisfies the same properties as the counit $S$.
 \end{proof}

We note that the properties of $\epsilon$ and $S$ cited in Theorem \ref{thm:BQG-definition} extend to the situation where $a$ is a multiplier and $b\in\sA$.

We can now define a bornological quantum group (with involution) by adding an involution to Voigt's definition \cite{Born}, and requiring positivity of the invariant integral.

\begin{defin}
 \label{def:BQG}
 A \emph{bornological quantum group algebra} is a bornological $*$-algebra $\sA$ satisfying the equivalent conditions of Theorem \ref{thm:BQG-definition}.
\end{defin}

As usual, we will use the notation $\sA=\sA(\mathbb{G})$ when the algebra is to be thought of as the algebra of functions on some  quantum group $\mathbb{G}$.

\subsection{Modular properties of the integral}

Let $\sA=\sA(\mathbb{G})$ be a bornological quantum group.
It is shown in \cite[Proposition 5.4]{Born}, following \cite[Proposition 3.8]{Vandaele}, that there exists a unique multiplier $\delta_{\mathbb{G}} \in\sM(\sA)$, called the \emph{modular element}, such that 
\[
(\phi\hat\otimes\id)(\Delta(a)) = \phi(a)\delta_{\mathbb{G}}
\]
for all $a\in\sA$.  It is group-like, so that
 \[
  \Delta(\delta_{\mathbb{G}}) = \delta_{\mathbb{G}}\otimes\delta_{\mathbb{G}}, \quad
  \epsilon(\delta_{\mathbb{G}}) = 1, \quad
  S(\delta_{\mathbb{G}}) = \delta_{\mathbb{G}}^{-1}.
 \]
Following the proof of \cite[Lemma 3.3]{KvD}, one sees that $\delta_{\mathbb{G}}$ is strictly positive in the sense that, for all nonzero $a\in\sA$,
 \begin{equation}
  \label{eq:delta_positive}
  \phi(a^*\delta_{\mathbb{G}} a) >0.
 \end{equation}
Hence $\delta_{\mathbb{G}}=\delta_{\mathbb{G}}^*$.

The Haar integral and the modular element are related by
\[
 \phi(a\delta_{\mathbb{G}}) = \phi(S(a))
\]
for all $a\in\sA$.  The proof of this is essentially the same as for \cite[Proposition 3.10]{Vandaele}.  Applying this twice gives $\phi(S^2(a))= \phi(\delta_{\mathbb{G}}^{-1}a\delta_{\mathbb{G}})$, and since $\phi\circ S^2$ is again a left-invariant integral, we have $\phi(S^2(a)) = \mu\phi(a)$ for some scalar $\mu\in\CC$, called the \emph{scaling constant}.  For algebraic quantum groups, De Commer and Van Daele have shown that we always have $\mu=1$, see \cite[Theorem, 3.4]{DeCVan}.  At present, we do not know if this is true for bornological quantum groups.  

To simplify the exposition, we will assume in this article that $\mu=1$, since all the examples we have in mind satisfy this assumption.   Thus we have $\phi(\delta_{\mathbb{G}} a) = \phi(a\delta_{\mathbb{G}})$ for all $a\in\sA$.  The situation $\mu\neq1$ would not add any particular difficulties, following the same methods as in \cite{KvD}.

There is a unique bounded algebra automorphism $\sigma:\sA\to\sA$ such that $\phi(ab) = \phi(b\sigma(a))$ for all $a,b\in\sA$, see \cite[Proposition 3.12]{Vandaele} and \cite[Proposition 5.3]{Born}.  This continues to hold when one of $a$ or $b$ is a multiplier, and by taking $b=1$ we have that $\phi$ is invariant under $\sigma$.  Our assumption that the scaling constant is $1$ implies that
\[
 \sigma(\delta_{\mathbb{G}}) = \delta_{\mathbb{G}}.
\]

We record some further basic properties of $\sigma$.

\begin{prop}
\label{prop:sigma_properties}
For all $a\in \sA$ we have
\begin{align*}
 \sigma(S(\sigma(a))) &= \delta_{\mathbb{G}}^{-1}S(a)\delta_{\mathbb{G}}, &
 \sigma^{-1}S(\sigma^{-1}(a))) &= \delta_{\mathbb{G}} S(a)\delta_{\mathbb{G}}^{-1}, \\
 S^2(\sigma(a)) &= \sigma(S^2(a)), &
 \sigma(\overline{a}) &= \overline{\sigma^{-1}(a)}.
\end{align*}
\end{prop}

\begin{proof}
 For any $a,b\in\sA$, we have
 \begin{align*}
  \phi(b \sigma(S(\sigma(a))))
   &= \phi( S(\sigma(a)) b) 
   = \phi(S^{-1}(b)\sigma(a)\delta_{\mathbb{G}}) \\
   &=  \phi(a \delta_{\mathbb{G}} S^{-1}(b))
   = \phi(b \delta_{\mathbb{G}}^{-1} S(a) \delta_{\mathbb{G}}).
 \end{align*}
 This proves the first equality.  The second follows by pre- and post-composing with $\sigma^{-1}$, and the third then follows by composing the first two.
 The final equality follows from
	\[
	 \phi(b\sigma(\overline{a})) = \phi(\overline{a} b) 
	  = \overline{\phi(\overline{b}a)} = \overline{\phi(\sigma^{-1}(a)\overline{b})}
	  = \phi(b\, \overline{\sigma^{-1}(a)}).
	\]
\end{proof}

The map $\sigma$ is not generally a coalgebra automorphism.  Instead, we have the following property, as in \cite{KvD}.
\begin{prop}
 \label{prop:sigma_coproduct}
 We have
\[
 \Delta \circ \sigma = (S^2 \otimes \sigma)\circ\Delta = (\sigma \otimes \alpha)\circ \Delta,
\]
where $\alpha$ is the bounded algebra automorphism defined by 
$
 \alpha(a) = \delta_{\mathbb{G}}^{-1}S^{-2}(a)\delta_{\mathbb{G}}.
$
\end{prop}

\begin{proof}
	Let $a,b,c\in\sA(\mathbb{G})$.  Using the invariance of the Haar integral, we have
	\begin{align*}
	 (\phi\hat\otimes\phi)((b\otimes c)\Delta(\sigma(a)))
	  &= \phi(b S(c_{(1)})) \phi (c_{(2)}\sigma(a)) \\
	  &= \phi(b S(c_{(1)})) \phi (ac_{(2)}) \\
	  &= \phi(b S^2(a_{(1)})) \phi(a_{(2)}c) \\
	  &= \phi(b S^2(a_{(1)})) \phi(c \sigma(a_{(2)})),
	\end{align*}
 which proves the first equality.  For the second, we calculate
	\begin{align*}
	 \phi((b\otimes c)\Delta(\sigma(a)))
	  &= \phi(b_{(1)}\sigma(a))\phi(c S^{-1}(b_{(2)})\delta_{\mathbb{G}}) \\
	  &= \phi(a b_{(1)}) \phi(b_{(2)} S(c)) \\
	  &= \phi(a_{(1)}b) \phi(S^{-1}(a_{(2)})\delta_{\mathbb{G}} S(c)) \\
	  &= \phi(b \sigma(a_{(1)})) \phi(c \delta_{\mathbb{G}}^{-1} S^{-2}(a_{(2)}) \delta_{\mathbb{G}}).
	\end{align*}
\end{proof}

One can also consider the automorphism $\sigma'$ associated to the right-invariant integral $\phi\circ S$, that is, $\phi(S(ab))=\phi(S(b\sigma'(a)))$. We get immediately that 
\begin{equation}
 \label{eq:sigma_prime}    \sigma'(a)=\delta_{\mathbb{G}}\sigma(a)\delta_{\mathbb{G}}^{-1} = \sigma(\delta_{\mathbb{G}} a \delta_{\mathbb{G}}^{-1})=S^{-1}(\sigma^{-1}(S(a))).
\end{equation}


\subsection{Pontryagin duality}

Let $\sA=\sA(\mathbb{G})$ be a bornological quantum group.  We write $\hat{a}$ or $\mathfrak{F}(a)$ for the bounded linear functional $\hat{a}:b\mapsto \phi(ba)$.
The Pontryagin dual, denoted $\hat{\sA}$ or $\sA(\hat{\mathbb{G}})$, is the space of bounded linear forms
\[
 \sA(\hat{\mathbb{G}}) = \{ \hat{a} \mid a\in\sA(\mathbb{G}) \} \subset \sA(\mathbb{G})^*
\]
equipped with the bornology inherited from the bijection $\mathfrak{F}:\sA \to \hat\sA$
and the Hopf operations defined by skew-duality, namely, for $a,b\in\sA$ and $x,y\in\hat{\sA}$
\begin{align*}
 (xy,a) &= (x\otimes y,\Delta(a)) &
 (\hat{\Delta}(x),a\otimes b) &= (x,ba) \\
 \hat\epsilon(x) &= (x,1) &
 (\hat{1},a) &= \epsilon(a) \\
 (\hat{S}(x),a) &= (x,S^{-1}(a)) &
 (\hat{S}^{-1}(x),a) &= (x,S(a)) \\
 (x^*,a) &= \overline{\left(x,\overline{S(a)}\right)}, &
 (x,\overline{a}) &= \overline{\left(\hat{S}^{-1}(x)^*,a\right)}.
\end{align*}
Note that we are using $\overline{a}$ for the involution of $a\in\sA(\GG)$ and $x^*$ for the involution of $x\in\sA(\hat\GG)$.
The left Haar integral $\hat\phi$ on $\sA(\hat{\mathbb{G}})$ is given by 
\[
 \hat{\phi}(\mathfrak{F}(a)) = \epsilon(a).
\]

The proof that $\sA(\hat{\mathbb{G}})$ is indeed a bornological quantum group  is due to Voigt \cite[Theorem 7.5]{Born}, with the exception of the $*$-structure.  We will confirm that the $*$-structure is compatible with the quantum group structure in Proposition \ref{prop:dual_involution} below.

Using the linear isomorphism $\sF$ we can transfer the Hopf operations from $\sA(\hat{\mathbb{G}})$ to $\sA(\mathbb{G})$, Specifically, we introduce the convolution product and convolution adjoint on $\sA(\mathbb{G})$,
\begin{align}
 f*g := f_{(1)}\phi(S^{-1}(g)f_{(2)}) & = \phi(S^{-1}(g_{(1)})f)g_{(2)}, 
 \label{eq:convolution}\\
 f^* =  & \overline{S(f)}\delta_{\mathbb{G}}.
 \label{eq:convolution_adjoint}
\end{align}
Then one can verify the following formulas for the dual operations:
\begin{align*}
 \mathfrak{F}(f)\mathfrak{F}(g) &= \mathfrak{F}(f*g),   &
 \mathfrak{F}(f)^* & = \mathfrak{F}(f^*), \\
 \hat\epsilon(\mathfrak{F}(f)) &= \phi(f), &
 \hat{S}(\mathfrak{F}(f)) &=\mathfrak{F}(\sigma(\delta_{\mathbb{G}}\,S(f))).
\end{align*}

\begin{prop}
 \label{prop:dual_involution}
 The involution $*$ defined on $\sA(\hat{\mathbb{G}})$ by the duality relations above makes $\sA(\hat{\mathbb{G}})$ into a bornological quantum group in the sense of Definition \ref{def:BQG}.
\end{prop}

\begin{proof}
Using \cite[Theorem 7.5]{Born}, we only need to check the compatibility of the involution.  We see from the formula \eqref{eq:convolution_adjoint} that the convolution adjoint maps $\sA(\mathbb{G})$ to $\sA(\mathbb{G})$, so the involution is well-defined on $\sA(\hat{\mathbb{G}})$.  The fact that $x\mapsto x^*$ is a bounded involutive antilinear algebra anti-automorphism and coalgebra auto-morphism is straightforward, and positivity of the left invariant integral $\hat{\phi}$ follows from the following well-known formula.
\end{proof}

\begin{lm}
  \label{lem:dual_ip}
  For any $f,g\in\sA(\mathbb{G})$ we have $\epsilon(f^* *g) = \phi(\overline{f}g)$.
\end{lm}

\begin{proof}
  We have $
   \epsilon(f^**g) = \epsilon(g_{(2)})\phi(S^{-1}(g_{(1)})S^{-1}(\overline{f})\delta_{\mathbb{G}})
   = \phi(S^{-1}(\overline{f}g)\delta_{\mathbb{G}}) = \phi(\overline{f}g).
  $
\end{proof}

We shall write $\sD(\mathbb{G})$ for the linear space $\sA(\mathbb{G})$ equipped with the Hopf operations pulled back from $\sA(\hat{\mathbb{G}})$ via $\sF$.  In particular, as a $*$-algebra, $\sD(\mathbb{G})$ is equipped with the convolution product \eqref{eq:convolution} and convolution adjoint \eqref{eq:convolution_adjoint} above, while the counit on $\sD(\mathbb{G})$ is $\hat{\epsilon}=\phi$ and the antipode on $\sD(\mathbb{G})$ is given by 
\[
 \hat{S}(f) = \sigma(\delta_{\mathbb{G}} S(f)).
\]
From this and Proposition \ref{prop:sigma_properties}, we get the notable formula
\begin{equation}
 \label{eq:Shat2}
 \hat{S}^2(f) =  S^2(f).
\end{equation}

We record the following compatibility between the pointwise coproduct and the convolution product.

\begin{lm}
	\label{lem:convolution_coproduct_compatibility}
	For any $f,g \in \sA(G)$ we have the formal equalities
	\[
	 \Delta(f*g) = f_{(1)} \otimes (f_{(2)}*g) = (f*g_{(1)})\otimes g_{(2)}.
	\]
	More precisely, for any $a\in\sA(G)$ we have
	\begin{align*}
	 (a\otimes 1)\Delta(f*g) &= af_{(1)} \otimes (f_{(2)}*g) &
	 \Delta(f*g)(a\otimes 1) &= f_{(1)}a \otimes (f_{(2)}*g) \\
   (1\otimes a)\Delta(f*g) &= (f*g_{(1)}) \otimes a g_{(2)} &
	 \Delta(f*g)(1\otimes a) &= (f*g_{(1)}) \otimes g_{(2)}a,
	\end{align*}
	where the right hand side of the first equation is understood by first applying a Galois map to $a\otimes f$ and then taking the convolution with $g$ in the second leg, and similarly for the others.
\end{lm}

\begin{proof}
  We calculate
  \begin{align*}
   (a\otimes 1) \Delta(f*g)
    &= (a\otimes 1) \Delta (f_{(1)}\, \phi(S^{-1}(g)f_{(2)})) \\
    &= (\id\hat\otimes\hat\id\otimes\hat\phi)(af_{(1)} \otimes f_{(2)} \otimes S^{-1}(g)f_{(3)}) \\
    &= (af_{(1)} \otimes f_{(2)}) * (\hat{1}\otimes g),
  \end{align*}
  where $\hat{1}$ denotes the unit in the convolution algebra $\sM(\sD(G))$.  The other equalities are similar.
\end{proof}

\subsection{Modular properties of the dual quantum group and Radford's $S^4$ formula}

From now on, we will write $\delta_{\mathbb{G}}$, $\sigma_{\mathbb{G}}$, \emph{etc.} for the modular element and modular automorphism of $\mathbb{G}$, and $\delta_{\hat{\mathbb{G}}}$, $\sigma_{\hat{\mathbb{G}}}$ for those of $\sD(\mathbb{G})\cong \sA(\hat{\mathbb{G}})$.
We can give explicit formulas for the modular automorphisms of  $\hat{\mathbb{G}}$.

\begin{prop}\label{mod}
Let $f\in \mathcal{D}(\mathbb{G})$.  We have
\begin{align*}
 \sigma_{\hat{\mathbb{G}}}(f)&=S^2(f)\delta_{\mathbb{G}}^{-1}, &
 \sigma'_{\hat{\mathbb{G}}}(f)&= \delta_{\mathbb{G}}^{-1}S^{-2}(f)\\
\end{align*}
\end{prop}

\begin{proof}
Let $f,g\in \mathcal{D}(\mathbb{G})$.  On the one hand we have
\begin{align*}
    \phi_{\hat{\mathbb{G}}}(f*g)&=\epsilon(f*g)
    =\phi_{\mathbb{G}}(S^{-1}(g)f)
\end{align*}
and on the other hand 
\begin{align*}
    \phi_{\hat{\mathbb{G}}}(g*(S^2(f)\delta_{\mathbb{G}}^{-1}))&=\phi_{\mathbb{G}}(S^{-1}(S^2(f)\delta_{\mathbb{G}}^{-1})g)\\
    &=\phi_{\mathbb{G}}(\delta_{\mathbb{G}}S(f)g)\\
    &=\phi_{\mathbb{G}}(S^{-1}(g)f),
\end{align*}
which leads to the first equality.  For the second equality we can dualize the identity \eqref{eq:sigma_prime} to obtain $ \sigma'_{\hat{\mathbb{G}}}=\hat{S}^{-1}\sigma_{\hat{\mathbb{G}}}^{-1}\hat{S}$. Hence, 
\begin{align*}
 \sigma'_{\hat{\mathbb{G}}}(f)
  &= S^{-1}(\delta_{\mathbb{G}}^{-1}\sigma_{\mathbb{G}}^{-1}(S^{-2}(\sigma_{\mathbb{G}}(\delta_{\mathbb{G}} S(f))\delta_{\mathbb{G}}))) \\
  &=  \delta_{\mathbb{G}}^{-1} S^{-3}(\delta_{\mathbb{G}}S(f)) \delta_{\mathbb{G}} \\
  &=  \delta_{\mathbb{G}}^{-1} S^{-2}(f).
\end{align*}
\end{proof}

Dualizing this formula and using Equation \eqref{eq:Shat2} yields the following.

\begin{cor}
\label{hatdelta}
For $f\in \mathcal{A}(\mathbb{G})$, we have 
\begin{align*}
 \sigma_{\mathbb{G}}(f)&= S^2(f)*\delta_{\hat{\mathbb{G}}}^{-1}, \\
 \sigma'_{\mathbb{G}}(f)&=\delta_{\hat{\mathbb{G}}}^{-1}*S^{-2}(f). 
\end{align*}
\end{cor}

\begin{prop}
 \label{prop:deltas_commute}
 The left and right actions of $\delta_\GG$ and $\delta_{\hat\GG}$ on $\sA(\GG)$ by multiplication and convolution, respectively, all commute.   
\end{prop}

\begin{proof}
 The fact that left and right multiplication by $\delta_\GG$ commute is obvious, as is the commutativity of left and right convolution by $\delta_{\hat\GG}$.  
 Using Corollary \ref{hatdelta} we have, for all $f\in\sA(G)$
 \begin{align*}
  \delta_{\hat\GG} * f &= S^{-2}({\sigma'_{\GG}}^{-1}(f)) ,
  &
  f * \delta_{\hat\GG} &= S^2(\sigma_\GG^{-1}(f)).
 \end{align*}
 Therefore, noting that $\sigma_\GG(\delta_\GG) = \sigma'_\GG(\delta_\GG) = \delta_\GG$, we obtain
 \begin{align*}
  \delta_{\hat\GG}*(\delta_\GG f)
   = S^{-2}({\sigma'_\GG}^{-1}(\delta_\GG f))
   = \delta_\GG (\delta_{\hat\GG}*f).
 \end{align*}
 The calculations for other combinations of actions are similar.
\end{proof}

\begin{rmk}
 \label{rmk:delta-commuting}
 If the scaling constant $\mu$ is not $1$, these operators will commute up to a scalar, and moreover left and right convolution by $\hat{\delta}_\GG$ will commute on the nose with the conjugation operator $f\mapsto \delta_\GG f \delta_\GG^{-1}$.  This shows that the proof of the next theorem remains valid even if the scaling constant is not $1$.  
\end{rmk}

One can now generalize Radford's $S^4$ formula to bornological quantum group.  See \cite{Radford-note} for a discussion about this formula in the algebraic case.

\begin{thm}(Radford's $S^4$ formula)
\label{thm:Radford}
Let $f\in \mathcal{A}(\mathbb{G})$, we have
$$S^4(f)=\delta_{\mathbb{G}}(\delta_{\hat{\mathbb{G}}}^{-1}*f*\delta_{\hat{\mathbb{G}}})\delta_{\mathbb{G}}^{-1}. $$
\end{thm}

\begin{proof}
Consider $g=S^2(f)$.  We have $\sigma_{\mathbb{G}}(g)=\delta_{\mathbb{G}}\sigma'_{\mathbb{G}}(g)\delta_{\mathbb{G}}^{-1}$ and thus 
$$ S^4(f)*\delta_{\hat{\mathbb{G}}}^{-1}=\delta_{\mathbb{G}}(\delta_{\hat{\mathbb{G}}}^{-1}*f)\delta_{\mathbb{G}}^{-1}. $$
Since the actions of $\delta_{\hat{\mathbb{G}}}$ and $\delta_{\mathbb{G}}$ commute, we are done.
\end{proof}

\subsection{The bornological multiplicative unitary}
\label{sec:curly_W}

Amongst the sixteen Galois maps and their inverses, one is particularly favoured.  This choice, called the \emph{multiplicative unitary}, comes from conventions fixed by Baaj and Skandalis in their foundational work on analytical quantum groups \cite{BaaSka}.  Here we give the bornological version.

\begin{defin}
The {bornological multiplicative unitary} is the linear bornological isomorphism
	\begin{align*}
	\sW = (\rho_l^\op)^{-1} : \sA(G) \btimes \sA(G) &\to \sA(G)\btimes\sA(G) \\
	a\otimes b &\mapsto 
	S^{-1}(b_{(1)})a \otimes b_{(2)}.
	\end{align*}
	with inverse
	\[
	\sW^{-1} : a\otimes b \mapsto 
	\Delta(b) (a\otimes 1) =
	b_{(1)}a\otimes b_{(2)}.
	\]
\end{defin}

\begin{prop}
	\label{prop:sW}
	The bornological multiplicative unitary is a unitary multiplier of the algebra $\sA(G)\btimes\sD(G)$, in the sense that
	\begin{equation}
	\label{eq:W_adjoint}
	(\sW(a\otimes b))^{*} \bullet (c\otimes d) = (a\otimes b)^{*} \bullet \sW^{-1}(c\otimes d) ,
	\end{equation}
	where $\bullet$ and $*$ denote the product and involution in $\sA(G)\btimes\sD(G)$. 
\end{prop}

\begin{proof}
	First, we check that $\sW^{-1}$ is right $\sA(G)\btimes\sD(G)$-linear.  Using Lemma \ref{lem:convolution_coproduct_compatibility}, we calculate
	\begin{align*}
	\sW^{-1} ((a\otimes b)\bullet (c\otimes d))
	&= \Delta(b*d)(ac\otimes 1) \\
	&= (\Delta(b)\cdot(1\otimes d))(ac\otimes 1) \\
	&= \sW^{-1}(a\otimes b) \bullet (c\otimes d).
	\end{align*}
	Thus $\sW$ is a left multiplier.  Using the fact that $\phi(\;\cdot\; \delta)$ is a right invariant integral, we obtain
	\begin{align*}
	( \sW^{-1} & (a\otimes b) )^* \bullet (\sW^{-1}(c\otimes d)) \\
	&= (\Delta(b)(a\otimes 1))^* \bullet (\Delta(d)(c\otimes 1)) \\
	&= \overline{a}\overline{b_{(1)}} d_{(1)} c \otimes (S^{-1}(\overline{b_{(2)}})\delta * d_{(2)}) \\
	&= \overline{a}\overline{b_{(1)}} d_{(1)} c \otimes  \phi(S^{-1}(d_{(2)}) S^{-1}(\overline{b_{(2)}} ) \delta) d_{(3)} \\
	&= \overline{a}S(S^{-1}(\overline{b_{(1)}} d_{(1)})) c \otimes  \phi(S^{-1}(\overline{b_{(2)}} d_{(2)} ) \delta) d_{(3)} \\
	&= \overline{a} c \otimes  \phi(S^{-1}(\overline{b_{(1)}} d_{(1)} ) \delta) d_{(2)} \\
	&= \overline{a} c \otimes  \phi(S^{-1}(d_{(1)}) b^*) d_{(2)} \\
	&= \overline{a} c \otimes b^* * d \\
	&= (a\otimes b)^* \bullet (c\otimes d).
	\end{align*}
	This proves that the left multiplier $\sW^{-1}$ admits $\sW$ as an adjoint in the sense of Equation \eqref{eq:W_adjoint}. It follows that $\sW$ is a two-sided multiplier, since we can define the associated right multiplier by
	\[
	 (a\otimes b)\cdot\sW = (\sW^{-1}\cdot (a\otimes b)^*)^*.
	\] 
	This completes the proof.
\end{proof}

The bornological multiplicative unitary $\sW$ satisfies the pentagonal equation
\[
\sW_{12}\sW_{13}\sW_{23} = \sW_{23}\sW_{12}
\]
and the bicharacter properties
\begin{align}
\label{eq:W_bicharacter}
  (\Delta \hat\otimes \id)\sW & = \sW_{13}\sW_{23}, &
  (\id \hat\otimes \hat\Delta)\sW &= \sW_{13}\sW_{12}, 
\end{align}


Let us record two further relations concerning the bornological multiplicative unitary.

\begin{lm}
	\label{lem:W_sigma_commutation}
	Considering $\sW$ as a linear automorphism of $\sA(G)\btimes\sA(G)$, we have
	\[
	(\sigma\hat\otimes\sigma)\sW = \sW(\sigma\hat\otimes\alpha),
	\]
	where $\alpha$ is the automorphism $\alpha:a \mapsto \delta^{-1} S^{-2}(a) \delta$ defined in Proposition \ref{prop:sigma_coproduct}.  Moreover,
	\[
	(\alpha\hat\otimes\alpha)\sW = \sW(\alpha\hat\otimes\alpha).
	\]
\end{lm}

\begin{proof}
	Let $a,b\in\sA(G)$.  According to Proposition \ref{prop:sigma_coproduct}, we have
	\begin{align*}
	\sW^{-1}(\sigma\hat\otimes\sigma)&(a\otimes b)
	= \Delta(\sigma(b))(\sigma(a)\otimes 1) \\
	&= (\sigma\hat\otimes\alpha)(\Delta(b)(a\otimes 1))
	= (\sigma\hat\otimes\alpha)\sW^{-1}(a\otimes b),
	\end{align*}
	which proves the first equality.  The second follows from the fact that $\alpha$ is a Hopf morphism (though not a Hopf *-morphism).  
\end{proof}

\section{From bornological to locally compact quantum groups}
\label{sec:lcqg}

\subsection{The left regular representation : Construction of $C_0^r(\mathbb{G})$ }
We fix a GNS pair $(L^2(\mathbb{G}),\Lambda)$ associated to $\phi_{\mathbb{G}}$. This means that $L^2(\mathbb{G})$ is a Hilbert space with a linear map $\Lambda : \mathcal{A}(\mathbb{G}) \rightarrow L^2(\mathbb{G})$ such that $\Lambda( \mathcal{A}(\mathbb{G}))$ is a dense subspace and we have
$$
\prodscal{\Lambda(f)}{\Lambda(g)}_{L^2(\mathbb{G})}=\phi_{\mathbb{G}}(\bar{f}g)=\epsilon(f^**g),\qquad \forall f,g\in \mathcal{A}(\mathbb{G}).
$$

\begin{rmk}
  The map $\Lambda : \mathcal{A}(\mathbb{G})\rightarrow L^2(\mathbb{G})$ is bounded with respect to the von Neumann bornology of $L^2(\mathbb{G})$, since the map $\|\cdot\|\circ\Lambda$, which maps $a\in \mathcal{A}(\mathbb{G})$ to $\phi_{\mathbb{G}}(a^*a)^{\frac{1}{2}}$, is bounded as the composition of bounded maps $a\mapsto a\otimes a^*\mapsto a^*a\mapsto \phi_{\mathbb{G}}(a^*a)^{\frac{1}{2}}$.
\end{rmk}

We denote by $m$ the left action of $\mathcal{A}(\mathbb{G})$ on $\Lambda(\mathcal{A}(\mathbb{G}))\subset L^2(\mathbb{G})$ by multiplication and by $\lambda$ the left action of $\mathcal{D}(\mathbb{G})$ by convolution, that is
\begin{itemize}
    \item $m(f)\Lambda(g)=\Lambda(fg)$,

 \item $\lambda(f)\Lambda(g)=\Lambda(f*g).$
 \end{itemize}
 Our first goal in this section is to show that densely defined operators $m(f)$, $ f\in\mathcal{A}(\mathbb{G})$ extend to bounded operators on $L^2(\mathbb{G})$. This will be done by looking at the multiplicative unitary on $L^2(\mathbb{G})\otimes L^2(\mathbb{G})$. First, note that $\Lambda\times \Lambda :  \mathcal{A}(\mathbb{G})\times \mathcal{A}(\mathbb{G})\rightarrow L^2(\mathbb{G})\otimes L^2(\mathbb{G})$ is a bounded bilinear map and thus extends to a bounded map $\Lambda\hat{\otimes} \Lambda : \mathcal{A}(\mathbb{G})\hat{\otimes} \mathcal{A}(\mathbb{G})\rightarrow L^2(\mathbb{G})\otimes L^2(\mathbb{G})$.

\begin{prop}
\label{prop:W}
There exists a unique unitary operator $W$ of $L^2(\mathbb{G})\otimes L^2(\mathbb{G})$ s.t. $W(\Lambda\hat\otimes\Lambda)(\Delta(g)(f\otimes 1))=\Lambda(f)\otimes\Lambda(g)$, for all $f,g\in  \mathcal{A}(\mathbb{G})$.
It is a multiplicative unitary on $L^2(\mathbb{G})$ in the sense that $W_{12}W_{13}W_{23}=W_{23}W_{12}$.
\end{prop}

\begin{proof}
First, by the hypothesis on the Galois maps, it is clear that this operator $W$ is well defined and invertible on $\Lambda\hat{\otimes} \Lambda (\mathcal{A}(\mathbb{G})\hat{\otimes} \mathcal{A}(\mathbb{G}))$. To check the unitarity let $a\otimes b$, $c\otimes d\in \mathcal{A}(\mathbb{G})\otimes \mathcal{A}(\mathbb{G})$ and observe that  
\begin{align*}
    \prodscal{\Delta(b)(a\otimes 1)}{\Delta(d)(c\otimes 1)}&=\phi_{\mathbb{G}}\hat\otimes \phi_{\mathbb{G}}((\bar{a}\otimes 1)\Delta(\bar{b}d)(c\otimes 1))\\
    &=\phi_{\mathbb{G}}(\bar{a}c) \phi_{\mathbb{G}}(\bar{b}d)\\
    &=\prodscal{a\otimes b}{c\otimes d}.
\end{align*}

Recall that the bornological multiplicative unitary $\sW$ belongs to $\sM(\sA(\GG)\btimes\sD(\GG))$.
From Lemma \ref{lem:dual_ip}, the inner product on $L^2(\GG)\otimes L^2(\GG)$ is given by 
\[
 \prodscal{\Lambda(a)\otimes\Lambda(b)}{\Lambda(c)\otimes \Lambda(d)} = (\phi_\GG \hat\otimes \epsilon)((\Lambda\btimes\Lambda)((\overline{a}\otimes b^*)\bullet(c\otimes d))),
\]
for all $a,c\in\sA(\GG)$, $b,d\in\sD(\GG)$, where $\bullet$ denotes the product in $\sA(\GG)\otimes\sD(\GG)$. 
It follows from Proposition \ref{prop:sW} that the densely defined operator
\[
 (m\otimes \lambda)(\sW) : \Lambda\btimes\Lambda (a\otimes b) \mapsto \Lambda\btimes\Lambda(\sW(a\otimes b)) 
\]
extends to a unitary operator $W$ on $L^2(\GG)$ with the stated properties.
\end{proof}

Given $\xi,\eta\in L^2(\GG)$, we denote by $\omega_{\xi,\eta}$ the state on $B(L^2(\GG))$ given by
\[
 \omega_{\xi,\eta}(T) = \prodscal{\xi}{T\eta}.
\]
This will allow us to define the left and right slices of the multiplicative unitary.

\begin{lm}
\label{lem:left-slice}
For any $f,g\in \mathcal{A}(\mathbb{G})$, the endomorphism $m((\iota\hat{\otimes}\phi_{\mathbb{G}})(\mathcal{W}^{-1}(f\otimes g)))$ of $\Lambda(\mathcal{A}(\mathbb{G}))$ extends to a bounded operator of $L^2(\mathbb{G})$. Explicitly, it extends to the left slice $(\iota\hat\otimes\omega_{\Lambda(\overline{g}),\Lambda(f)})(W)$.
\end{lm}
\begin{proof}
A straightforward calculation, as in
\cite[Lemma 2.3]{KvD}, 
shows that for all $h\in \mathcal{A}(\mathbb{G})$ we have
\begin{equation}
 \label{fond}
 (\iota\hat\otimes\omega_{\Lambda(f),\Lambda(g)})(W)\Lambda(h)=\Lambda((\iota\hat\otimes\phi_{\mathbb{G}})(\Delta(\overline{f})(1\otimes g))h).
\end{equation}
On the right hand side we have $m(\mathcal{W}^{-1}(g\otimes \overline{f}))\Lambda(h)$ and on the other side $(\iota\hat\otimes\omega_{\Lambda(f),\Lambda(g)})(W)$, which is a bounded operator.
\end{proof}

\begin{prop}
The left regular representation \normalfont $m : \mathcal{A}(\mathbb{G})\rightarrow \textrm{End}(\Lambda(\mathcal{A}(\mathbb{G})))$ extends to a bounded $*$-representation $m : \mathcal{A}(\mathbb{G}) \rightarrow B(L^2(\mathbb{G}))$.
\end{prop}
\begin{proof}
The bilinear map $(f,g)\mapsto (\iota\hat\otimes\omega_{\Lambda(f),\Lambda(g)})(W)$ from $\mathcal{A}(\mathbb{G}))\times \mathcal{A}(\mathbb{G})$ into $B(L^2(\mathbb{G}))$ is clearly bounded. Thus it extends to $\mathcal{A}(\mathbb{G})\hat{\otimes}\mathcal{A}(\mathbb{G})$. Let $x\in \mathcal{A}(\mathbb{G}) $ such that $\phi_{\mathbb{G}}(x)=1$. For all $a\in \mathcal{A}(\mathbb{G})$, using Lemma \ref{lem:left-slice}, one can obtain $m(a)$ as the composition 
$$a\mapsto a\otimes x \mapsto \mathcal{W}(a\otimes x) \stackrel{m\circ (\iota\hat\otimes\phi_{\mathbb{G}})\circ \mathcal{W}^{-1}}{\longmapsto} m(a).$$
\end{proof}

\begin{defin}
We define the reduced $C^*$-algebra of functions on $\GG$, denoted $C_0^r(\mathbb{G})$, as the closure of $m(\mathcal{A}(\mathbb{G}))$ in $B(L^2(\mathbb{G}))$.
\end{defin}

\begin{prop}
\label{prop:W-slices}
We have that  $\{(\iota\hat\otimes\omega_{\Lambda(f),\Lambda(g)})(W) \mid f,g\in \mathcal{A}(\mathbb{G}) \}=m(\mathcal{A}(\mathbb{G}))$.\end{prop}
\begin{proof}
We have that $\mathcal{W}^{-1}$ is an isomorphism of $\mathcal{A}(\mathbb{G})\hat{\otimes} \mathcal{A}(\mathbb{G})$ into itself and because $\iota\hat\otimes\phi_{\mathbb{G}} : \mathcal{A}(\mathbb{G})\hat{\otimes} \mathcal{A}(\mathbb{G})\rightarrow \mathcal{A}(\mathbb{G})$ is surjective we obtain that $\mathcal{A}(\mathbb{G})=\{(\iota\hat\otimes\phi_{\mathbb{G}})((\Delta(\overline{a})(1\otimes b)) \mid  a,b\in \mathcal{A}(\mathbb{G})\} $.  Thus the result follows from Equation \eqref{fond}.
\end{proof}

We also derive from (\ref{fond}) the following result 

\begin{prop}
\label{c*}
 The $C^*$-algebra $C_0^r(\mathbb{G})$ is the norm closure in $B(L^2(\mathbb{G}))$ of $\{(\iota\hat\otimes\omega)(W) \mid \omega\in B(L^2(\GG))_*\}$.
\end{prop}

\begin{defin}
We define the mapping $\Delta$ from $C_0^r(\mathbb{G})$ into $B(L^2(\mathbb{G})\otimes L^2(\mathbb{G}))$ such that $\Delta(x)=W^*(1\otimes x)W$.
\end{defin}

The proof of the following result can be readily adapted from the proof of the corresponding result in the algebraic case \cite[Theorem 2.11]{KvD}.

\begin{thm}
 We have that $C_0^r(\mathbb{G})$ is a non-degenerate $C^*$-subalgebra of $B(L^2(\mathbb{G}))$ and $\Delta$ is a non-degenerate
injective $*$-homomorphism from $C_0^r(\mathbb{G})$ to $M(C_0^r(\mathbb{G})\otimes C_0^r(\mathbb{G}))$ such that:
\begin{itemize}
    \item  $(\Delta\hat\otimes\iota)\circ \Delta=(\iota\hat\otimes\Delta)\circ \Delta$
\item  The vector spaces $\Delta(C_0^r(\mathbb{G}))(C_0^r(\mathbb{G})\otimes 1)$ and $\Delta(C_0^r(\mathbb{G}))(1\otimes C_0^r(\mathbb{G}))$ are dense subsets of $C_0^r(\mathbb{G})\otimes C_0^r(\mathbb{G})$.
\end{itemize}
\end{thm}

A similar construction yields the regular representation $\lambda$ of $\sD(\mathbb{G})$, as follows.

\begin{prop}
    \label{prop:lambda}
    For any $x\in\sD(\mathbb{G})$, $\lambda(x)$ extends to a bounded operator on $L^2(\mathbb{G})$.  Explicitly, if $f,g\in\sA(\mathbb{G})$ we have
    \[
     (\omega_{\Lambda(f),\Lambda(g)}\hat\otimes\id)W = \lambda(g\,\sigma_\mathbb{G}(\overline{f})).
    \]
    The resulting map $\lambda:\sD(\mathbb{G}) \to B(L^2(\mathbb{G}))$ is a bounded $*$-representation.
\end{prop}

\begin{proof}
  This is another standard calculation.  For any $a,b\in\sA(\mathbb{G})$ we have
  \begin{align*}
      \prodscal{\lambda(a)}{( (\omega_{\Lambda(f),\Lambda(g)}\hat\otimes\id)W )\Lambda(b)}
       &= \prodscal{\Lambda\hat\otimes\Lambda(f\otimes a)}{W\Lambda\hat\otimes\Lambda(g\otimes b)} \\
       &= (\phi_\GG\hat\otimes\phi_\GG)(S^{-1}(b_{(1)})g \sigma_{\mathbb{G}}(\overline{f}) \otimes \overline{a}b_{(2)}) \\
       &= \prodscal{\Lambda(a)}{\Lambda((g \sigma_{\mathbb{G}}(\overline{f}))*b)},
  \end{align*}
  which proves the displayed formula.
  Since $\sA(\mathbb{G})$ is essential, it follows that $\lambda(x)$ extends to a bounded operator for every $x\in\sD(\mathbb{G})$.
\end{proof}

Note that, from the definition of $W$ in Proposition \ref{prop:W}, the bornological and $C^*$-algebraic multiplicative unitaries can now be related by $W = (m\otimes\lambda)(\sW)$.

\begin{defin}
 We define the $C^*$-algebra $C^*_r(G)$ as the norm closure of $\{(\omega\hat\otimes\iota)(W) \mid \omega\in B(L^2(\mathbb{G}))_*\}$.
\end{defin}

A standard calculation shows that the bornological multiplicative unitary for the Pontryagin dual $\hat\GG$ is given by $\hat{\sW} = \flip(\sW^*)$, where $*$ denotes the involution of $\sA(\GG)\hat\otimes\sA(\hat\GG) \cong \sA(\GG)\hat\otimes\sD(\GG)$.  Moreover, the map $\sF : \sA(\GG) \mapsto \sA(\hat\GG)$ extends to an isometric isomorphism of $L^2(\GG)$ with $L^2(\hat\GG)$ thanks to Lemma \ref{lem:dual_ip}.  Using this, we obtain the following result, which should be no surprise.

\begin{prop}
 We have that $C^*_r(\GG)=C_0^r(\hat{\mathbb{G}})$.
\end{prop}

\subsection{The modular element at the $C^*$-algebraic level}

In order to extend $\delta_{\mathbb{G}}$ to a positive operator on $L^2(\mathbb{G})$ we shall introduce another GNS construction.  For the inspiration here, see \cite[Section 3]{KvD}. 

Recall from Equation \eqref{eq:delta_positive} that $\delta_\GG\in\sM(\sA(\GG))$ is strictly positive:
\[
 \phi_\GG(\overline{a}\delta_\GG a) >0 \qquad \text{for all nonzero } a\in\sA(\GG).
\]
We can therefore define a Hilbert space $L^2(\mathbb{G})_{\delta}$ together with an injective linear map $\Lambda_{\delta}$ from $\mathcal{A}(\mathbb{G})$ to $L^2(\mathbb{G})_{\delta}$ such that 
\begin{enumerate}[label=\arabic*.]
    \item $\Lambda_{\delta}$ has dense range in $L^2(\mathbb{G})_{\delta}$,
    \item $\prodscal{\Lambda_{\delta}(f)}{\Lambda_{\delta}(g)}=\phi_{\mathbb{G}}(\overline{f}\delta_{\mathbb{G}}g)$ for all $f,g\in \mathcal{A}(\mathbb{G})$.
\end{enumerate}

We  now define the closed operator $L$ from $L^2(\mathbb{G})$ to $L^2(\mathbb{G})_{\delta}$ with core $\Lambda(\mathcal{A}(\mathbb{G}))$ such that for every $f\in \mathcal{A}(\mathbb{G})$ we have $L\Lambda(f)=\Lambda_{\delta}(f)$. Then
$$\prodscal{Lv}{\Lambda_{\delta}(f)}=\prodscal{v}{\Lambda(\delta_{\mathbb{G}}f)} $$
for any $v\in  \Dom(L)$ and $f\in \mathcal{A}(\mathbb{G})$. It follows that $\Lambda_{\delta}(\mathcal{A}(\mathbb{G}))$ is a subset of $\Dom(L^*)$ and that $L^*\Lambda_{\delta}(f)=\Lambda(\delta_{\mathbb{G}}f)$.

\begin{defin}
    We set $\delta=L^*L$, so that $\delta$ is a positive operator on  $L^2(\mathbb{G})$. Note that for all $f\in \mathcal{A}(\mathbb{G})$
    $$\delta\Lambda(f)=\Lambda(\delta_{\mathbb{G}}f).$$
    We denote by $\hat{\delta}$ the  operator associated to $\delta_{\hat{\mathbb{G}}}$ via the analogous construction.
\end{defin}

We now recall a technical lemma that will be used regularly in the rest of this chapter.   For the proof see \cite[Lemma 3.7]{KvD}.
\begin{lm}\label{tech}

Consider Hilbert spaces $K_1$, $K_2$, $H_1$, $H_2$, a unitary operator $U$ from $K_1$ to $H_2$, a unitary
operator $V$ from $
H_1$ to $K_2$, a closed linear operator $F$ from within $K_1$ into $H_1$, a closed linear operator $G$ 
from within $H_2$ into $K_2$.
Suppose there exists a core $C$ for $F$ such that $U(C)$ is a core for $G$ and such that  $ V(F(v)) = G(U(v))$ for every $v\in C$. Then we have that $VF = GU$.
\end{lm}
\begin{lm}
 The operator $U$ from $L^2(\mathbb{G})\hat\otimes L^2(\mathbb{G})_{\delta}$ to $L^2(\mathbb{G})_{\delta}\hat\otimes L^2(\mathbb{G})_{\delta}$ such that $U(\Lambda(f)\otimes \Lambda_{\delta}(g)) = (\Lambda_{\delta}\hat\otimes \Lambda_{\delta})(\Delta(g)(f\otimes1))$ is well defined and unitary.
\end{lm}
\begin{proof}
Direct calculation.
\end{proof}

\begin{lm}
\label{lem:W_delta_relation}
We have $(1\otimes \delta)W=W(\delta\otimes\delta)$.
\end{lm}
\begin{proof}

 Let $f,g\in \mathcal{A}(\mathbb{G})$, we have
 \begin{align*}
     (L\hat\otimes L)W^*(\Lambda(f)\otimes \Lambda(g))&=(L\hat\otimes L)(\Lambda\hat\otimes\Lambda)(\Delta(g)(f\otimes1))\\
     &=(\Lambda_{\delta}\hat\otimes\Lambda_{\delta})(\Delta(g)(f\otimes1))\\
     &=U(1\hat\otimes L)(\Lambda(f)\otimes \Lambda(g)).
 \end{align*}
Using lemma \ref{tech} we deduce that $(L\hat\otimes L)W^*=U(1\hat\otimes L)$.  Composing this with its adjoint, the result follows.
\end{proof}

\begin{prop}
 \label{prop:delta_affiliated}
 We have that $\delta$ is a strictly positive element affiliated with $C_0^r(\mathbb{G})$ in the $C^*$-algebraic sense. Furthermore, $\Delta(\delta)=\delta\otimes \delta$.
\end{prop}

\begin{proof}
Our proof is very similar that the proofs of \cite[Propositions 8.5 and 8.6]{KvD}.
From the preceding lemma, we obtain that
$$
(1\otimes \delta^{-it})W(1\otimes \delta^{it})=(\delta^{it}\otimes 1)W, 
$$
for all $t\in\RR$.
Let $\omega\in \mathcal{K}(L^2(\mathbb{G}))^*$.  Applying  $\iota\hat\otimes \omega$ to this equality we get 
$$\delta^{it}((\iota\hat\otimes \omega)W)=\iota\hat\otimes  \delta^{-it}\omega \delta^{it}(W),$$
where the notation $\delta^{-it}\omega \delta^{it}$ refers to the functional $\omega(\delta^{-it}\cdot \delta^{it})$. Thus, by Proposition \ref{c*} we conclude that $\delta^{it}C_0^r(\mathbb{G})\subset C_0^r(\mathbb{G})$. 

We also derive from Lemma \ref{lem:W_delta_relation} that 
$W^*(1\otimes \delta)W=\delta\otimes \delta$, i.e. $\Delta(\delta)=\delta\otimes \delta$.
\end{proof}

By induction on $n$, one can deduce the following lemma.

\begin{lm}
 Consider $f\in \mathcal{A}(\mathbb{G})$ and $n\in \mathbb{Z}$. Then $\Lambda(f)$ belongs to $\Dom(\delta^n)$ and $\delta^n\Lambda(f)=\Lambda(\delta_{\mathbb{G}}^nf)$.
\end{lm}

\begin{lm}
 Consider $f\in \mathcal{A}(\mathbb{G})$ and $z\in \mathbb{C}$. Then $\Lambda(f)$ belongs to $\Dom(\delta^z)$.
\end{lm}

\begin{proof}
 We already saw in Proposition \ref{prop:delta_affiliated} that $\Lambda(f)\in\Dom(\delta^{it})$ for all $t\in\RR$, so an interpolation using the previous lemma proves the result.
\end{proof}

Similarly, the proof of \cite[Lemma 8.9]{KvD} is still valid for the following two propositions.  Here we are writing $\sDom(T) \subseteq C_0^r(\mathbb{G})$ for the domain of a positive element $T$ affiliated to the $C^*$-algebra $C_0^r(\mathbb{G})$.

\begin{prop}
 For every $n\in \mathbb{Z}$ and $f\in \mathcal{A}(\mathbb{G}))$, we have that $m(f)$ belongs to $\sDom(\delta_{\mathbb{G}}^n)$ and $\delta^nm(f)=m(\delta_{\mathbb{G}}^nf)$.
\end{prop}
\begin{prop}
 For every $z\in \mathbb{C}$ and $f\in \mathcal{A}(\mathbb{G}))$, we have that $m(f)$ belongs to $\sDom(\delta^z)$.
\end{prop}

As in the algebraic framework, we will prove more: that the complex powers $\delta^z$ of the $C^*$-algebraic modular element multiply the bornological subalgebra $m(\sA(\GG))$ into itself and so define bounded multipliers of $\sA(\GG)$ in the bornological sense.  To do so, we need a series of technical lemmas.

Firstly, we observe that by Pontryagin duality, elements of the bornological  dual $\sA(\hat\GG)$ also give elements of the pre-dual.

\begin{lm}
\label{lem:Ghat-predual}
For every $f\in \mathcal{A}(\mathbb{G})$, the linear functional  $\hat{f}=\mathcal{F}(f)\in \sA(\mathbb{G})^*$ extends to a bounded linear functional on $B(L^2(\mathbb{G}))$ and we obtain a bounded linear map $\mathcal{A}(\mathbb{G})\rightarrow B(L^2(\mathbb{G}))_*$.
\end{lm}
\begin{proof}
The linear map 
\begin{align*}
    \mathcal{A}(\mathbb{G})\otimes \mathcal{A}(\mathbb{G}) &\longrightarrow  B(L^2(\mathbb{G}))_*\\
    a\otimes b~~~~ &\longmapsto \widehat{ab}=\omega_{\Lambda(\sigma(b^*)),\Lambda(a)}
\end{align*}
is bounded so extends to  $ \mathcal{A}(\mathbb{G})\hat{\otimes} \mathcal{A}(\mathbb{G})$. Precomposing with the isomorphism $\mathcal{A}(\mathbb{G})\cong\mathcal{A}(\mathbb{G})\hat{\otimes}_{\mathcal{A}(\mathbb{G})} \mathcal{A}(\mathbb{G})$ yields the result.
\end{proof}

The next Lemma is the bornological analogue of \cite[Lemma 7.6]{KvD}.  It essentially says that slices of the $C^*$-algebraic coproduct by elements of the bornological dual yield bornological multipliers.

\begin{lm}\label{alg}
 Consider $f,g\in \mathcal{A}(\mathbb{G})$ and $x\in M(C_0^r(\mathbb{G}))$, then $(\iota\otimes \hat{f})(\Delta(x))m(g)$ belongs to $m(\mathcal{A}(\mathbb{G}))$.
\end{lm}

\begin{proof}
We adapt the proof of \cite[Lemma 7.6]{KvD} to the bornological context. Let $x\in M(C_0^r(\mathbb{G}))$ and consider the linear map $L_x : \mathcal{A}(\mathbb{G}))\hat{\otimes}\mathcal{A}(\mathbb{G})\rightarrow B(L^2(\mathbb{G}))$ defined by $f\otimes g\mapsto (\iota\otimes \hat{f})(\Delta(x))m(g)$. 

Let $\Delta(q)(r\otimes 1)\in \mathcal{A}(\mathbb{G})\hat{\otimes}\mathcal{A}(\mathbb{G})$. For every $y\in \mathcal{A}(\mathbb{G})$, we have
\begin{align}
    L_{m(y)}(\Delta(q)(r\otimes 1))&=m((\iota\hat\otimes \phi_{\mathbb{G}})(\Delta(y)\Delta(q)(r\otimes 1))) \nonumber\\
    &=\phi_{\mathbb{G}}(yq)m(r). \label{eq:alg} \nonumber
\end{align}
Using Lemma \ref{lem:Ghat-predual}, and because
$m(\mathcal{A}(\mathbb{G}))$ is strictly dense in $M(C_0^r(\mathbb{G}))$, one can replace $m(y)$ by any $x\in M(C_0^r(\mathbb{G}))$ in this equality. Therefore $L_x\otimes \gamma_l$ is bounded and extends to $\mathcal{A}(\mathbb{G})\hat{\otimes}\mathcal{A}(\mathbb{G})$.  Thus $L_x$ maps $\mathcal{A}(\mathbb{G})\hat{\otimes}\mathcal{A}(\mathbb{G})$ into $\mathcal{A}(\mathbb{G})$ as required. 
\end{proof}

Now we return to the complex powers of the modular element.

\begin{lm}
 \label{lem:delta-z-lemma}
 Let $z\in\mathbb{C}$ and $f\in \mathcal{A}(\mathbb{G})$, The linear map $\iota\hat\otimes (\hat{f}\circ\delta^z\circ m) : \mathcal{A}(\mathbb{G})\hat{\otimes} \mathcal{A}(\mathbb{G}) \rightarrow  \mathcal{A}(\mathbb{G})$ is well-defined, bounded and surjective.
\end{lm}

\begin{proof}
We must show that the bilinear map $\mathcal{A}(\mathbb{G})\otimes \mathcal{A}(\mathbb{G}) \to \sA(\GG)$, $a\otimes b\mapsto a\hat{f}(\delta^zm(b))$ is bounded. It is sufficient to check that $b\mapsto \hat{f}(\delta^zm(b))$ is bounded.  For that we observe that for any bounded subset $B$ of $\mathcal{A}(\mathbb{G})$, $m(B)$ is a bounded subset of $B(L^2(\mathbb{G}))$ and consists of analytic elements for $\delta$ so $\hat{f}(\delta^zm(B))$ is a bounded subset of $\mathbb{C}$. Thus this map can be extended to $\mathcal{A}(\mathbb{G})\hat{\otimes} \mathcal{A}(\mathbb{G})$ and corresponds to the map given in the Lemma. Surjectivity is immediate.
\end{proof}

\justify Next we recall Lemma 8.11 of \cite{KvD}, for which the proof also remains valid.

\begin{lm}\label{tech2}
Consider an element $\alpha$ affiliated with $C_0^r(\mathbb{G})$ and elements $x\in \sDom(\alpha)$, $y\in C_0^r(\mathbb{G})$. Then $\Delta(x)(1\otimes y)$ belongs to $\sDom(\Delta(\alpha))$ and $\Delta(\alpha)\Delta(x)(1\otimes y)=\Delta(\alpha(x))(1\otimes y)$.
\end{lm}

\begin{prop}
 Let $z \in \mathbb{C}$.  Then $\delta^zm(\mathcal{A}(\mathbb{G}))\subset m(\mathcal{A}(\mathbb{G}))$ and $\delta^z$ is a bounded multiplier of $m(\mathcal{A}(\mathbb{G}))$, where $m(\mathcal{A}(\mathbb{G}))$ is endowed with the bornology inherited from $\sA(G)$ through the injective linear map $m$.
\end{prop}

\begin{proof}
In this proof we adapt the arguments of the proof of \cite[Proposition 8.12]{KvD}. Let $\Delta(p)(q\otimes 1)\in \mathcal{A}(\mathbb{G})\hat{\otimes} \mathcal{A}(\mathbb{G})$ and $f\in\mathcal{A}(\mathbb{G})$. We consider the element $a=(\iota\hat\otimes(\hat{f}\circ\delta^z\circ m))(\Delta(p)(q\otimes 1))$ of $\mathcal{A}(\mathbb{G})$.  We have
\begin{align*}
    \delta^zm(a)&=(\iota\hat\otimes\hat{f})((\delta^z\otimes \delta^z)(\Delta(m(p))(m(q)\otimes 1))\\
    &=(\iota\hat\otimes\hat{f})(\Delta(\delta^zm(p)))(m(q)\otimes 1).
\end{align*}
Since $m(p)$ is an analytic element for $\delta$, we have that $\delta^zm(p)\in M(C_0^r(\mathbb{G}))$. By Lemma \ref{alg}, it follows that $\delta^zm(a)$ belongs to $m(\mathcal{A}(\mathbb{G}))$. Now, because of the boundedness of the map $(\iota\hat\otimes\hat{f}\circ\delta^z\circ m)$, see Lemma \ref{lem:delta-z-lemma}, we can apply this method to any element $a=(\iota\hat{\otimes}\hat{f}\delta^z\circ m)({X})$, ${X}\in \mathcal{A}(\mathbb{G})\hat{\otimes} \mathcal{A}(\mathbb{G})$, that is, to any element $a$ of $\mathcal{A}(\mathbb{G})$. Thus one can now define $\delta_\GG^{-z}$ as the unique multiplier of $\mathcal{A}(\mathbb{G})$ such that 
\[
 m(\delta_{\GG}^z a) = \delta^z m(a).
\]

To prove that it is indeed a bounded multiplier, let  $a\in \mathcal{A}(\mathbb{G})$ and consider an element $b\in \mathcal{A}(\mathbb{G})$ such that $\phi_{\mathbb{G}}(\delta^{-z}_{\mathbb{G}}b)=1$. We have $\Delta(\delta^{-z}_{\mathbb{G}}b)=(\delta^{-z}_{\mathbb{G}}\otimes\delta^{-z}_{\mathbb{G}})\Delta(b)$. It follows that\\ 
$$a=a\phi_{\mathbb{G}}(\delta^{-z}_{\mathbb{G}}b)=(\iota\hat{\otimes}\phi_{\mathbb{G}})(\Delta(\delta^{-z}_{\mathbb{G}}b)(a\otimes1)) =\delta^{-z}_{\mathbb{G}}(\iota\hat{\otimes}\phi_{\mathbb{G}}\delta^{-z}_{\mathbb{G}})(\Delta(b)(a\otimes1)),$$
where $\phi_\GG\delta_\GG^{-z}$ denotes the linear functional $g \mapsto \phi_\GG(\delta_\GG^{-z}g)$.
Thus, the multiplier $\delta^{z}_{\mathbb{G}}$ can be expressed as the composition map 
$$a\mapsto a\otimes b\mapsto \Delta(b)(a\otimes1)\mapsto (\iota\hat\otimes\phi_{\mathbb{G}}\delta^{-z}_{\mathbb{G}})(\Delta(b)(a\otimes1)). $$
It remains to show that the last map in this composition is well defined and bounded. Let $x,y$ and $f$ in $\sA(\GG)$, we have
\begin{align*}
    (\iota\otimes\phi_{\mathbb{G}}\delta^{-z}_{\mathbb{G}})(xf\otimes y)&=(\iota\otimes\phi_{\mathbb{G}}\sigma^{-1}(f)\delta^{-z}_{\mathbb{G}})(x\otimes y)\\
    &=(\iota\otimes (\widehat{\sigma^{-1}(f)}\circ\delta^z\circ m))(x\otimes y)
\end{align*}
We then deduce the boundedness of $(x,y)\mapsto(\iota\otimes\phi_{\mathbb{G}}\delta^{-z}_{\mathbb{G}})(x\otimes y)$  using the essentiality of $\sA(\GG)$.
\end{proof}

With the above proposition, the following theorem is now straightforward, compare \cite[Section 8]{KvD}.

\begin{thm}
 \label{thm:delta-z}
 For all $z\in \mathbb{C}$, there exists a unique bounded mutliplier of $\mathcal{A}(\mathbb{G})$ denoted $\delta^z_{\mathbb{G}}$ such that for all $a\in \mathcal{A}(\mathbb{G})$, 
$$\delta^zm(a)=m(\delta^z_{\mathbb{G}}a).$$ Furthermore, we have the following properties : 
 \begin{enumerate}
     \item For any $z\in \mathbb{C}$, $\overline{\delta^z_{\mathbb{G}}}=\delta^{\bar{z}}_{\mathbb{G}}$
     \item For any $y,z\in \mathbb{C}$, $\delta^y_{\mathbb{G}}\delta^z_{\mathbb{G}}=\delta^{y+z}_{\mathbb{G}}$,
     \item For any $t\in\mathbb{R}, \delta^{it}_{\mathbb{G}}$ is unitary in $\mathcal{M}(\mathcal{A}(\mathbb{G}))$,
     \item For any $t\in\mathbb{R}$, $\delta^{t}_{\mathbb{G}}$ is a positive element, in the sense that $\delta^{t}_{\mathbb{G}}=\delta^{t/2}_{\mathbb{G}}\delta^{t/2}_{\mathbb{G}}$ and $\delta^{t/2}_{\mathbb{G}}$ is a self adjoint element.
     \end{enumerate}
\end{thm}

With our assumption that the scaling constant is $1$, we obtain the following.

\begin{prop}
 The right Haar functionnal $\phi_{\mathbb{G}}\circ S$ of $\mathcal{A}(\mathbb{G})$ is positive.
\end{prop}
\begin{proof}
Let $f\in \mathcal{A}(\mathbb{G})$, we have
$$\phi_{\mathbb{G}}(S(\bar{f}f))=\phi_{\mathbb{G}}(\bar{f}f\delta_{\mathbb{G}})=\phi_{\mathbb{G}}(\overline{f\delta^{1/2}_{\mathbb{G}}}f\delta^{1/2}_{\mathbb{G}})>0, $$
where we use that $\delta^{1/2}_{\mathbb{G}}$ is self-ajdoint. 
\end{proof}

\subsection{Preliminary remarks on the modular group}

A central point in the framework of locally compact quantum groups \cite{KV:LCQG} is a good understanding of the modular theory of the associated operator algebras.  Let us briefly recall the main definitions.

First, we define the closed operator $T$ on $L^2(\GG)$ as the closed antilinear operator with core $\Lambda(\sA(\GG))$ such that $T\Lambda(f) = \Lambda(\overline{f})$ for all $f\in\sA(\GG)$.  It can be checked that $T^*\Lambda(f) = \Lambda(\sigma_\GG(\overline{f}))$ for all $f\in\sA(\GG)$.  Hence the modular operator $\nabla = T^*T$ satisfies
\[
 \nabla \Lambda(f) = \Lambda(\sigma_\GG(f)).
\]
We denote by $J$ the anti-unitary component of the polar decomposition of $T$, so that $T=J\nabla^{\frac12}= \nabla^{\frac12}J$.

\begin{defin}
\label{def:modular_group}
 Let $x\in B(L^2(\GG)$.  We define $\sigma_t(x) = \nabla^{it}x\nabla^{-it}$.  The family $(\sigma_t)_{t\in\RR}$ is called the modular group associated to $C^r_0(\GG)$.
\end{defin}

Classically, the study of the modular group is undertaken using the unitary antipode and the scaling group $\tau_t(x) = M^{it}xM^{-it}$, where $M$ is the positive operator in the polar decomposition $G=IM^{\frac12}$ of the closed antilinear operator $G$ with $G\Lambda(f) = \Lambda(S(\overline{f}))$, see \cite{KvD, KV:LCQG}.  (Kustermans and Vaes \cite{KV:LCQG} use $N$ for the operator $M$, but since we are mainly following \cite{KvD} here, we will stick with their notation.)  In order to study the stability properties of an algebraic quantum group with respect to these operator algebraic automorphism groups, Kustermans and Van Daele proceed as for the complex powers of the modular element $\delta$ in the previous section, namely they seek out commutation relations between the positive operators $M$, $\nabla$ (and other auxiliary operators) and the multiplicative unitary $W$, in order to obtain similar relations for the associated automorphism groups.  

We shall follow the same general strategy, but with a change of focus. 
Note that, by Proposition \ref{mod}, the modular operator $\hat\nabla$ for the Pontryagin dual satisfies 
\begin{equation}
\label{eq:nabla-hat}
 \hat\nabla \Lambda(f) = \Lambda(S^2(f)\delta_\GG^{-1})
\end{equation}
for all $f\in\sD(\GG)$.  
This indicates that one can relate the modular group of the Pontryagin dual $(\hat{\sigma}_t)_{t\in\RR}$ to the complex powers of the modular element $\delta_\GG$, which we have already studied, and the automorphism group associated to a closure $N$ of the operator $S^2$. 
This can then be related to the usual scaling group and unitary antipode by the above formula for $\hat\nabla$, or its dual version.

The advantage of this approach is that the operator $S^2$ is both an algebra and a coalebra automorphism of $\sA(\GG)$, so has very nice algebraic properties.

\subsection{The automorphism group associated to $S^2$}

We denote by $(\Lambda',L^2(\mathbb{G})'$) the GNS construction associated to $(\mathcal{A}(\mathbb{G}),\phi_{\mathbb{G}}\circ S)$. In order to build a positive operator associated with $S^2$ we introduce the following operator.

\begin{defin}
We define $K$ as the closed unbounded antilinear operator from $L^2(\mathbb{G})$ to $L^2(\mathbb{G})'$ such that $\Lambda(\mathcal{A}(\mathbb{G}))$ is a core for $K$ and $K\Lambda(f)={\Lambda'}(S(\overline{f}))$. 
\end{defin}

\begin{lm}
 Let $f\in \mathcal{A}(\mathbb{G})$.  We have that $K^*\Lambda'(f)=\Lambda(\overline{S(f)})$
\end{lm}
\begin{proof}
 Let $f,g\in \mathcal{A}(\mathbb{G})$, we have 
 \begin{align*}
     \prodscal{K\Lambda(f)}{\Lambda'(g)}&=\phi_{\mathbb{G}}(S(\overline{S(\overline{f}})g))
     =\phi_{\mathbb{G}}(S(g)f)
     = \langle{\Lambda(\overline{S(g)})},{\Lambda{f}}\rangle
 \end{align*}
\end{proof}

\begin{defin}
    We set $N=K^*K$. 
\end{defin}

Thus $N$ is a positive operator on $L^2(\mathbb{G})$ such that $N\Lambda(f)=\Lambda(S^2(f))$ for all $f\in \mathcal{A}(\mathbb{G})$.  We remark (again) that this $N$ differs from the operator $N$ in \cite{KV:LCQG}, which corresponds to the operator denoted by $M$ here and in \cite{KvD}.  The operators $M$ and $N$ are very closely related, see the discussion after Proposition \ref{modular}.

\begin{lm}
 The operator $V$ from $L^2(\mathbb{G})'\hat\otimes L^2(\mathbb{G})'$ to $L^2(\mathbb{G})'\hat\otimes L^2(\mathbb{G})'$ such that $V^*(\Lambda'(f)\otimes \Lambda'(g)) = (\Lambda'\hat\otimes \Lambda')(\Delta^{cop}(g)(f\otimes1))$ is well defined and unitary.
\end{lm}
\begin{proof}
 Direct calculation.
\end{proof}

\begin{lm}
 We have that $(N\hat\otimes N)W=W(N\hat\otimes N)$.
\end{lm}
\begin{proof}
 First we prove that $(K\hat\otimes K)W^*=V^*(K\hat\otimes K)$ :\\
 Let $f,g\in \mathcal{A}(\mathbb{G})$, we have
 \begin{align*}
     (K\hat\otimes K)W^*(\Lambda(f)\otimes \Lambda(g))&=(K\hat\otimes K)(\Lambda\hat\otimes \Lambda)(\Delta(g)(f\otimes1))\\
     &=(\Lambda'\hat\otimes \Lambda')(\Delta^{cop}(S(\overline{g})(\overline{f}\otimes1))\\
     &=V^*(K\hat\otimes K)(\Lambda'(f)\otimes \Lambda'(g)).
 \end{align*}
 Using Lemma \ref{tech} we get that $(K\hat\otimes K)W^*=V^*(K\hat\otimes K)$. Similarly we also get that $(K^*\hat\otimes K^*)W^*=V^*(K^*\hat\otimes K^*)$ and the result follows.
\end{proof}

\begin{defin}
Let $(\rho_t)_{t\in\mathbb{R}}$ designate the one parameter group of automorphims $B(L^2(\mathbb{G}))$ generated by $N$, that is, for all $t\in\mathbb{R}$ and $x\in B(L^2(\mathbb{G}))$, we define $\rho_t(x)=N^{it}xN^{-it}$. 
\end{defin}

We will see that the automorphism group $(\rho_t)_{t\in\RR}$ is closely related to the scaling group $(\tau_t)_{t\in\RR}$.

\begin{prop}
For all $t\in\mathbb{R}$ we have that $\rho_t(C_0^r(\mathbb{G}))\subset C_0^r(\mathbb{G})$.
\end{prop}

\begin{proof}
From $(N\hat\otimes N)W=W(N\hat\otimes N)$ we  obtain
$$(N^{it}\hat\otimes1)W(N^{-it}\hat\otimes 1)=(1\hat\otimes N^{-it})W(1\hat\otimes N^{it})$$
Let $\omega\in {B(L^2(\GG))_*}$.   Applying  $\iota\hat\otimes \omega$ to this equality we get
$$N^{it}(\iota\hat\otimes\omega)(W)N^{-it}=(\iota\hat\otimes N^{-it}\omega N^{it})(W), $$
so the result follows from Proposition \ref{c*}.
\end{proof}

\begin{lm}
The  operators $\delta$ and $\hat{\delta}$ strongly commute with $N$.
\label{delta-N}
\end{lm}

\begin{proof}
Thanks to Lemma \ref{tech}, in order to show that $\delta$ and $N$ strongly commute it is enough to show that $\delta^{it}$ and $N$ commute on $\Lambda(\mathcal{A}(\mathbb{G}))$ for any $t\in\mathbb{R}$. Since $\delta^{it}_{\mathbb{G}}$ is a group-like element of $\sM(\mathcal{A}(\mathbb{G}))$ we have that $S^2(\delta^{it}_{\mathbb{G}})=\delta^{it}_{\mathbb{G}}$, and thus for all $f\in \mathcal{A}(\mathbb{G})$ we have 
\begin{align*}
 N\delta^{it}\Lambda(f)&=\Lambda(S^2(\delta^{it}_{\mathbb{G}}f))\\
 &=\Lambda(\delta^{it}_{\mathbb{G}}S^2(f))\\
 &=\delta^{it}N\Lambda(f).   
\end{align*}
A similar argument applies for $\hat{\delta}$.  Note that $\hat{S}^2(f) = S^2(f)$ for all $f\in\sD(\GG)$.
\end{proof}

We define $\delta'$ as the unbounded operator $\delta'=J\delta J$.
This is merely a convenient way to introduce the appropriate unbounded closure of the operator of right multiplication by $\delta_\GG$, since one can show that for all $f\in \mathcal{A}(\mathbb{G})$,
$$\delta'\Lambda(f)=\Lambda(f\delta_{\mathbb{G}}).$$
Dually, we make the analogous definition of the operator $\hat\delta'$, so that $$\hat\delta'\Lambda(f)=\Lambda(f*\delta_{\hat{\mathbb{G}}}).$$

The following then follows in an analogous fashion to Lemma \ref{delta-N}.

\begin{lm}
 \label{delta-prime-N}
 The operators $\delta'$ and $\hat{\delta}'$ strongly commute with $N$.
\end{lm}

In order to define $\rho_z$ at the bornological level, which informally can be understood as the operator $(S^2)^{z/2}$, we will generalize Radford's $S^4$ formula.  For this we need the following lemmas.\\

\begin{lm}
 \label{lem:nu}
 There exists a constant $\nu>0$ such that $\sigma(\delta_{\mathbb{G}}^{z})=\nu^{iz}\delta_{\mathbb{G}}^{z}$ for all $z\in\mathbb{C}$.
\end{lm}

\begin{proof}
First, applying Proposition \ref{prop:sigma_coproduct} to $\delta_\GG^z$, we derive that 
$$
\sigma(\delta_{\mathbb{G}}^{z})\otimes\delta_{\mathbb{G}}^{z}=\delta_{\mathbb{G}}^{z}\otimes \sigma(\delta_{\mathbb{G}}^{z})
$$
and thus there exists $c(z)\in \mathbb{C}$ such that $ \sigma(\delta_{\mathbb{G}}^{z})=c(z)\delta_{\mathbb{G}}^{z}$.  Clearly, $c:\CC\to\CC^\times$ is a homomorphism.   

 Now consider $f\in \mathcal{A}(\mathbb{G})$ such that $\phi_{\mathbb{G}}(\overline{f}f)=1$.  We have
 \begin{align*}
 \prodscal{\delta'^{-z}\Lambda(f)}{\delta'^{-z}\Lambda(f)}&=\phi_{\mathbb{G}}(\delta_{\mathbb{G}}^{z}\overline{f}f\delta_{\mathbb{G}}^{-z})\\
 &=c(z)\phi_{\mathbb{G}}(\overline{f}f),
 \end{align*}
 which is holomorphic.  Using Proposition \ref{prop:sigma_properties}, for $t\in\RR$ we have
\begin{align*}
    \sigma(\delta_{\mathbb{G}}^{it})&=\sigma\left(\overline{\delta_{\mathbb{G}}^{-it}}\right)\\
    &=\overline{\sigma^{-1}(\delta_{\mathbb{G}}^{-it})}\\
\end{align*}
 and thus $c(it)=\overline{c(it)}$. The result follows.
\end{proof}

\begin{lm}
 \label{lem:conjugations_commute}
 The operators
 $\delta{\delta'}^{-1}$ and $\hat\delta \hat{\delta'}^{-1}$ on $L^2(\GG)$ strongly commute.
\end{lm}

\begin{proof}
In fact, we will prove a stronger statement, namely that $\hat{\delta}$ and $\hat{\delta'}$ both strongly commute with $\delta{\delta'}^{-1}$.
From the proof of Proposition \ref{prop:deltas_commute} we have that $f*\delta_\mathbb{\hat{G}}=S^2(\sigma_{\mathbb{G}}^{-1}(f))$. Now, using the preceding lemma, let $t\in \mathbb{R}$ and observe that
\begin{align*}
    \hat{\delta}'\delta^{it}{\Lambda}(f)&=\Lambda((\delta^{it}_{\mathbb{G}}f)*\delta_\mathbb{\hat{G}})\\
    &=\Lambda(S^2(\sigma_{\mathbb{G}}^{-1}(\delta_{\mathbb{G}}^{it}f)))\\
    &=\nu^t \Lambda(\delta_{\mathbb{G}}^{it}S^{2}(\sigma_{\mathbb{G}}^{-1}(f))\\
    &=\nu^t \delta^{it}\hat{\delta}'{\Lambda}(f).
\end{align*}
In the same way, we have
\[
 \hat{\delta}'{\delta'}^{-it}{\Lambda}(f) = \nu^{-t}{\delta'}^{-it}\hat{\delta}'{\Lambda}(f).
\]
Combining these, we see that $\hat{\delta}'$ strongly commutes with $\delta{\delta'}^{-1}$. 

A similar argument shows that $\hat{\delta}$ strongly commutes with $\delta{\delta'}^{-1}$.
\end{proof}

\begin{rmk}
 Lemma \ref{lem:nu} holds even if the scaling constant is not $1$.  In fact, if the scaling constant is $1$ then we obtain $\nu=e^{2\pi k}$ for some $k\in\ZZ$, while in the general case we simply have $\nu\in(0,\infty)$.  Thus, even if the scaling constant is non trivial, we still obtain the strong commutation of the operators $\delta{\delta'}^{-1}$ and $\hat\delta\hat{\delta'}^{-1}$ as in Lemma \ref{lem:conjugations_commute}.  As previously, we do not know whether it is possible to have $\nu\neq1$ for a bornological quantum group.
\end{rmk}

\begin{thm}
 \label{thm:S2z}
 Let $z\in\mathbb{C}$, for any $f\in \mathcal{A}(\mathbb{G})$ we have that $\rho_z(m(f))$ belongs to $m(\mathcal{A}(\mathbb{G}))$. More precisely we have
 $$\rho_z(m(f))=m(\delta_{\mathbb{G}}^{-iz/2}(\delta_{\hat{\mathbb{G}}}^{iz/2}*f*\delta_{\hat{\mathbb{G}}}^{-iz/2})\delta_{\mathbb{G}}^{iz/2}). $$
\end{thm}

\begin{proof}
Considering Radford's $S^4$ formula, Theorem \ref{thm:Radford}, together with the fact that the real powers of all the modular operators are strongly commuting positive operators, we deduce that for any $g\in \mathcal{A}(\mathbb{G})$,
$$N^z \Lambda(g) =\Lambda(\delta_{\mathbb{G}}^{iz/2}(\delta_{\hat{\mathbb{G}}}^{-iz/2}*{g}*\delta_{\hat{\mathbb{G}}}^{iz/2})\delta_{\mathbb{G}}^{-iz/2}).$$
The result follows.
\end{proof}

\begin{rmk}
\label{C*-tau}
We note that the formula for $N^z\Lambda(g)$ in the proof is self-dual, up to a sign.  It follows that the bornological subalgebra $\lambda(\sD(\GG)) \subset C^*_r(\GG)$ is also stable with respect to the automorphism group $(\rho_z)_{z\in\CC}$.
\end{rmk}

\subsection{The modular groups of $C^*_r(\mathbb{G})$ and $C_0^r(\mathbb{G})$}
\label{sec:modular_groups}

As mentioned above, our approach for the construction of the modular group of $C_0^r(\mathbb{G})$ is somewhat different from that of \cite{KvD}.  We start by building the the modular group of $C^*_r(\mathbb{G})$ and then apply duality to get that of $C_0^r(\mathbb{G})$. The motivation for this is Equation \eqref{eq:nabla-hat}, which gives a formula for $\hat{\nabla}$ in terms of the strongly commuting operators $N$ and $\delta'$.  Thus, the modular group can be expressed in terms of the automorphism groups $(\rho_z)$ and $({\delta'}^z)$, both of which we have already shown to stabilize $\sA(\GG)$.

\begin{defin}
Let $x\in B(L^2(\mathbb{G}))$, we define ${\sigma}_t(x)={\nabla}^{it}x{\nabla}^{-it}$. The family $({\sigma}_t)_{t\in\mathbb{R}}$ is called the modular group associated to $C_0^r(\mathbb{G})$ and the functional $\phi_{\mathbb{G}}$. (We will see in Section \ref{sec:Haar_weight} that $\phi_{\mathbb{G}}$ can be extended into a left Haar weight for $C_0^r(\mathbb{G})$.) 

Similarly we introduce the modular group associated to the dual $\hat{\sigma}_t$, given by $\hat{\sigma}_t(x)=\hat{\nabla}^{it}x\hat{\nabla}^{-it}$.
\end{defin}

Recall that we use $\lambda:\sD(\GG) \to C^*_r(\GG)$ to denote the regular representation.  We may extend $(\hat{\sigma}_t)$ to a complex $1$-parameter group on analytic elements.

\begin{prop}
 Let $f\in \mathcal{D}(\mathbb{G})$, and $n\in \mathbb{Z}$.  Then $\hat{\sigma}_{in}(\lambda(f))=\lambda(S^{-2n}(f)\delta_{\mathbb{G}}^{n}).$
\end{prop}

\begin{proof}
{A direct calculation using the above formula for $\hat\nabla$ shows that, for all $g\in\sA(G)$, 
\begin{align*}
 \hat{\sigma}_i(\lambda(f))\Lambda(g) 
   &= \hat\nabla^{-1}\lambda(f)\hat{\nabla}\Lambda(g) \\
   &= \Lambda(S^{-2}(f_{(1)})\delta_\GG)\phi_\GG(\delta_\GG S(g) f_{(2)}) \\
   &= \Lambda(S^{-2}(f_{(1)})\delta_\GG)\phi_\GG( S^{-1}(g) S^{-2}(f_{(2)}) \delta_\GG) \\
   &= \lambda(S^{-2}(f\delta_\GG) \Lambda(g).
\end{align*}
The result follows by induction.
}
\end{proof}

A standard interpolation argument allows us to conclude that the elements of $\lambda(\sD(\GG))$ are analytic for the modular group $(\hat\sigma_z)_{z\in\CC}$.  Moreover, since $N$ and $\delta'$ strongly commute, we have 
\[
 \sigma_z(\lambda(f)) = \delta'^{-iz}\rho_z(\lambda(f))\delta'^{iz},
\]
for all $f\in\sD(\GG)$.  Applying Theorem \ref{thm:delta-z} for the dual group $\hat\GG$ and Remark \ref{C*-tau}, we obtain the following result.

\begin{prop}
 \label{prop:sigma-hat-stability}
  We have $\hat{\sigma}_z(\lambda(\mathcal{D}(\mathbb{G})))\subset \lambda(\mathcal{D}(\mathbb{G}))$.
\end{prop}

By duality, one can deduce the analogous result for $\sigma_z$.

\begin{prop}\label{modular}
We have $\sigma_z(m(\mathcal{A}(\mathbb{G})))\subset m(\mathcal{A}(\mathbb{G}))$.
\end{prop}

Finally, although we shall not need it here, let us remark that the bornological algebra $\sA(\GG)$ is preserved by the scaling group $(\tau_z)_{z\in\CC}$.  Indeed, the scaling group is given by $\tau_t(x) = M^{-it}xM^{it}$, where $M$ is defined as a closure of the operator $\Lambda(f) \mapsto \Lambda(S^2(f)\delta) = \delta'N\Lambda(f)$, see \cite{KV:LCQG}.  Therefore, using the strong commutativity of the operators $N$ and $\delta'$, the stability of $\sA(\GG)$ by $\tau_t$ follows from the results above.

\subsection{A Left Haar weight for $(C_0^r(\mathbb{G}),\Delta)$}
\label{sec:Haar_weight}
The fact that the Haar functional $\phi_{\mathbb{G}}$ can be extended into a Haar weight of $C_0^r(\mathbb{G})$ is not trivial and, in the algebraic case, this is the whole consideration of \cite[Section 6]{KvD}. Here we follow the ideas of that section.

To begin we recall the following result of \cite[Section 6]{KvD}. This result can be directly applied in our case because it uses only common properties shared by algebraic and bornological quantum groups. This result uses the standard machinery of Hilbert algebras, which we will appeal to without comment.  For details, we refer the reader to \cite{KvD} and to the books \cite{Dixmier:vN, Takesaki}. 

\begin{prop}
There exists a faithful lower semi-continuous weight of $C_0^r(\mathbb{G})$, denoted $\phi$, such that $m(\mathcal{A}(\mathbb{G}))$ is a subset of $\mathcal{N}_{\phi}$ and $\phi(m(f))=\phi_{\mathbb{G}}(f)$ for all $f\in \mathcal{A}(\mathbb{G})$. Moreover we have that $\phi$ is invariant under $\sigma$ and more generally, $\Lambda_{\phi}(\sigma_t(x))=\nabla^{it}\Lambda_{\phi}(x).$
\end{prop}

Next, we relate this construction more specifically with the bornological structure.

\begin{prop}
 The set $m(\mathcal{A}(\mathbb{G}))$ is a core for $\Lambda_{\phi}$.
\end{prop}

\begin{proof}
The linear map $\Lambda_0 : m(\mathcal{A}(\mathbb{G}))\rightarrow L^2(\mathbb{G}), m(f)\mapsto \Lambda(f)$ satisfies $\Lambda_0\leq \Lambda_{\phi}$.  Thus it is closable and we again denote $\Lambda_0$ its closure, with domain denoted by $A_0$. Our goal is to show that $A_0=\mathcal{N}_{\phi}$. 

First, we observe that $A_0$ is a left ideal.
Let $a\in A_0$ and $x\in C_0^r(\mathbb{G})$. Because of the closedness of $\Lambda_0$, one can choose a sequence $a_n\in m(\mathcal{A}(\mathbb{G}))$ such that $(a_n)$ converges to $a$ in $C_0^r(\mathbb{G})$ and $(\Lambda_0(a_n))_n$ converges to $\Lambda_0(a)$. We also consider a sequence $x_n\in m(\mathcal{A}(\mathbb{G}))$ that converges to $x$. The sequence $(x_na_n)$ converges to $xa$ and for all $n$ we have $\Lambda_0(x_na_n)=x_n\Lambda_{\phi}(a_n)=x_n\Lambda_0(a_n)$. Thus $(\Lambda_0(x_na_n))_n$ is  convergent and so $xa$ belongs to $A_0$.

Now let $x\in \mathcal{N}_{\phi}$ and consider an approximate unit $(e_n)$ in $\mathcal{A}(\mathbb{G})$. The sequence $(xm(e_n))_n$ converges toward $x$ and each $xm(e_n)$ belongs to $A_0$. We have 
$$
\Lambda_0(xm(e_n))=\Lambda_{\phi}(xm(e_n))=J\sigma_{i/2}(m(e_n^*))J(\Lambda_{\phi}(x)), 
$$
where we recall that $J$ denotes the anti-unitary component of the polar decomposition of $T$ such that $T\Lambda(f) = \Lambda(\overline{f})$, see Section \ref{sec:modular_groups}.

The sequence $(\sigma_{i/2}(m(e_n)))$ converges strictly toward $1$ and thus $(\Lambda_0(xm(e_n)))$ converges toward $\Lambda_{\phi}(x)$. Thus $A_0=\mathcal{N}_{\phi}$, so $m(\mathcal{A}(\mathbb{G}))$ is a core for $\Lambda_{\phi}$.
\end{proof}

The proof of the following proposition is directly adapted from that of \cite[Lemma 6.4]{KvD}.

\begin{lm}
 Consider $x,y\in m(\mathcal{A}(\mathbb{G}))$ and $\omega\in C_0^r(\mathbb{G})^*$.  We have that 
 $(y\omega\hat\otimes \iota)(\Delta(x))$ belongs to $m(\mathcal{A}(\mathbb{G}))$ and
 $(y\omega\hat\otimes \phi)(\Delta(x))=\omega(y)\phi(x)$, where we are using the notation $y\omega$ for the functional $y\omega:a\mapsto \omega(ay)$.
\end{lm}

\begin{proof}
Let $f,g \in \mathcal{A}(\mathbb{G})$. We have 
\begin{align*}
    (m(g)\omega\hat\otimes \iota)(\Delta(m(f)))&=(\omega\hat\otimes \iota)(\Delta(m(f))(m(g)\otimes 1))\\
    &=(\omega\hat\otimes \iota)(m\hat\otimes m)(\Delta(f)(g\otimes 1))\\
    &=m((\omega\circ m\hat\otimes \iota)(\Delta(f)(g\otimes 1))).
\end{align*}
Note that the last equality rests on the fact that $(\omega\circ m\otimes \iota)$ is a bounded map and thus $(\omega\circ m\otimes \iota)(\Delta(f)(g\hat\otimes 1))$ is well defined and belongs to $\mathcal{A}(\mathbb{G})$. One can then conclude using the left invariance of $\phi_{\mathbb{G}}$.
\end{proof}

The following theorem corresponds to \cite[Theorem 6.13]{KvD}, giving the left invariance of the weight $\phi$.

\begin{thm}
Let $x\in \mathcal{M}_{\phi}$, we have that $(\omega\hat\otimes \iota)(\Delta(x))$ belongs to $\mathcal{M}_{\phi}$ and $(\omega\hat\otimes \phi_{\mathbb{G}})(\Delta(x))=\omega(1)\phi(x)$.  
\end{thm}

\begin{proof}
The proof can be directly adapted from that of \cite[Theorem 6.13]{KvD}, taking into account that $(\iota\hat\otimes \Lambda_{\phi})$ is bounded.
\end{proof}

\subsection{$C_0^r(\mathbb{G})$ as a reduced $C^*$-algebraic quantum group}

It remains to show that the left Haar weight is KMS.  In the context of algebraic quantum groups, Kustermans and Van Daele \cite{KvD} show the KMS property directly.  Kustermans and Vaes \cite{KV:LCQG} have since showed that approximately KMS suffices.  By definition, this means we must show that for a dense subset of elements $v\in L^2(\GG)$, there is a constant $M=M_v$ such that $\|xv\|_{L^2(\GG)} \leq M\|\Lambda(x)\|_{L^2(\GG)}$ for all $x\in\mathcal{N}_\phi$, see \cite[Definitions 1.33 and 1.34]{KV:LCQG}.

\begin{prop}
 The Haar state $\phi$ is an approximate KMS state.
\end{prop}
\begin{proof}
Let $a\in \mathcal{A}(\mathbb{G})$. For all $x\in\mathcal{N}_\phi$ and $w\in L^2(\mathbb{G})$ we have 
\begin{align*}
    \prodscal{\Lambda_{{\phi}}({x}a)}{w}&=\prodscal{T\Lambda_{{\phi}}(a^*{x^*})}{w}\\
    &=\prodscal{J\sigma_{i/2}(m(a^*))J\Lambda_{{\phi}}({x})}{w}.
\end{align*}
Thus, using Proposition \ref{modular}, we have $\|{x}\Lambda(a)\|\leq  \|\sigma_{i/2}(m(a^*)) \|_{B(L^2(\mathbb{G}))}\|\Lambda_{{\phi}}({x})\|$.  The result follows.
\end{proof}

The following result can now be checked immediately.  As usual, we use the notation $R$ for the unitary antipode, that is, the unitary closure of the densely defined operator $\tau_{i/2}\circ S$.  See \cite[Section 9]{KvD} for details.

\begin{prop} The weight
 $\phi\circ R$ is a positive right Haar weight for $(C_0^r(\mathbb{G}),\Delta)$ and satisfies the KMS condition with modular group $\sigma'=R\sigma R$.
\end{prop}

We now know that our quantum group $(C_0^r(\mathbb{G}),\Delta)$ satisfies the definition of a reduced $C^*$-algebraic quantum group given in \cite[Section 4]{KV:LCQG}.

\begin{thm}
The pair $(C_0^r(\mathbb{G}),\Delta)$ is a reduced $C^*$-algebraic quantum group.
\end{thm}

\subsection{Von Neumann, Fourier and universal algebras}
\label{sec:other_algebras}

Now that we have obtained a reduced $C^*$-algebraic quantum group $C_0^r(\GG)$ from a bornological quantum group $\GG$, the general theory of locally compact quantum groups \cite{KV:LCQG, Kus3} proovides us also with the von Neumann algebraic quantum group $L^\infty(\GG)$, the Fourier algebra $\sA(\GG)$, and the universal $C^*$-algebra of functions $C_0^u(\GG)$.  These algebras are necessary for defining homomorphisms and closed quantum subgroups in the locally compact framework, which we shall treat in the next section.

It should be no surprise that the bornological algebra $\sA(\GG)$ is dense in each of these, for the appropriate topologies.  To complete this section, we make the necessary remarks to confirm this.

The fact that $\sA(\GG)$ is weak operator dense in $L^\infty(\GG)$ is obvious, since $L^\infty(\GG)$ is the weak operator closure of $C^*_r(\GG)$.

For the Fourier algebra, we start with the dual object $\sA(\hat{\GG})=L^1(\GG)$, which is defined as the predual of $L^\infty(\GG)$.

\begin{prop}
	\label{prop:DG_in_L1G}
	The convolution algebra $\sD(\mathbb{G})$ is a dense in the convolution algebra $L^1(\mathbb{G})$.  Explicitly, for every $x\in\sD(\mathbb{G})$, the linear functional
	\[
	 \hat{x} : \sA(\mathbb{G}) \to \CC; \qquad \hat{x}:a\mapsto \phi(ax)
	\]
	extends to an ultraweakly continuous linear functional on $L^\infty(\mathbb{G})$, and the map $x\mapsto \hat{x}$ is a bounded $*$-algebra homomorphism of $\sD(\mathbb{G})$ into $L^1(\mathbb{G})$ with dense range in the Banach topology.
\end{prop}

\begin{proof}
	First consider $x=fg$ where $f,g\in\sA(\mathbb{G})$ and the product is the pointwise product of $\sA(\mathbb{G})$.  Then 
	\[
	 \hat{x}(a) = \phi(afg) = \phi(\sigma^{-1}(g)af) 
	   = \prodscal{ \Lambda(\sigma(\overline{g})) }{ m(a) \Lambda(f) }
	\]
	for all $a\in\sA(\mathbb{G})$, and this obviously extends to an element of the predual of $L^\infty(\mathbb{G})$.  Moreover, the sequence of maps
	\[
	  \begin{array}{ccccc}
	  \sA(\mathbb{G})\otimes\sA(\mathbb{G}) & \longrightarrow & L^2(\mathbb{G})\otimes L^2(\mathbb{G}) &
	    \longrightarrow & L^\infty(\mathbb{G})_* \qquad\qquad\\
	  f\otimes g & \longmapsto & \Lambda(\sigma(\overline{g})) \otimes \Lambda(f) &
	    \longmapsto & \prodscal{\Lambda(\sigma(\overline{g})) }{ m(\;\bullet\;) \Lambda(f) }
	  \end{array}
	\]
	is bounded, so extends to a bounded linear map
	\(
	 \sA(\mathbb{G}) = \sA(\mathbb{G})\btimes_{\sA(\mathbb{G})}\sA(\mathbb{G}) \to L^1(\mathbb{G}).
	\)

	Suppose now that $a\in L^\infty(\mathbb{G})$ is such that $\hat{x}(a) = 0$ for all $x\in\sD(\mathbb{G})$.  Then the above calculations show that $\prodscal{\Lambda(g)}{ m(a) \Lambda(f)} =0$ for all $f,g\in\sA(\mathbb{G})$.  But $\Lambda(\sA(\mathbb{G}))$ is dense in $L^2(\mathbb{G})$, so we get $a=0$.  This proves that the image of $\sD(\mathbb{G})$ in $L^1(\mathbb{G})$ is dense.
	
	The convolution products on $\sD(\mathbb{G})$ and its image in $L^1(\mathbb{G})$ clearly coincide because both are dual to the product in $\sA(\mathbb{G}) \subset L^\infty(\mathbb{G})$.
\end{proof}

  By duality, the bornological algebra $\sA(\GG)$ is dense in the Fourier algebra $A(\GG) = L^\infty(\hat{\GG})_*$.  We will frequently use the notation $x\mapsto\tilde{x}$ for the inclusion of $\sD(\GG)$ into $L^1(\GG)$, and likewise $a\mapsto \tilde{a}$ for the inclusion of $\sA(\GG)$ into $A(\GG)$.

    Finally, the universal $C^*$-algebraic quantum group $C_0^u(\GG)$ is the enveloping $C^*$-algebra of $A(\GG)$.  We will write $m^u_\GG$ for the universal representation of $\sA(\GG)$, namely,
    \[
     m^u_\GG:\sA(\GG) \to A(\GG) \to C_0^u(\GG).
    \]
    This is an injective $*$-algebra map with dense range.  Dually, we write
    \[
     \lambda^u_\GG:\sD(\GG) \to C^*_u(\GG)
    \]
    for the universal representation of the convolution algebra.  We may also consider elements of the Fourier algebra $A(\GG)$, or indeed the dense subalgebra $\sA(\GG)$, as forms on $C^*_u(\GG)$, by precomposing with the regular representation $C^*_u(\GG) \to C^*_r(\GG)$.

\section{Homomorphisms and closed quantum subgroups}
\label{sec:subgroups}

One of the major advantages of bornological quantum groups is the simplicity of the notion of a quantum subgroup.  In this section we define closed quantum subgroups of bornological quantum groups, and show that they give rise to closed quantum subgroups of the corresponding locally compact quantum groups.

\subsection{Morphisms of bornological quantum groups}

\begin{defin}
	Let $\mathbb{G}$ and $\mathbb{H}$ be bornological quantum groups.  A \emph{morphism of bornological quantum groups} from $\mathbb{H}$ to $\mathbb{G}$ is an essential $*$-algebra morphism $\pi:\sA(\mathbb{G}) \to \sM(\sA(\mathbb{H}))$ which intertwines the coproducts:
	\[
	 \Delta_{\mathbb{H}}\circ\pi = (\pi\hat\otimes\pi)\circ\Delta_{\mathbb{G}}
	\]

	If $\pi$ maps $\sA(\mathbb{G})$ surjectively onto $\sA(\mathbb{H})$, then we call $\mathbb{H}$ a \emph{closed quantum subgroup} of $\mathbb{G}$.  In this case we write $\pi=\pi_{\mathbb{H}}$ and refer to it as the \emph{restriction map}.
\end{defin}

Any morphism $\pi:\sA(\mathbb{G}) \to \sM(\sA(\mathbb{H}))$ of bornological quantum groups automatically respects the antipode and counit: 
	\[ 
	 S_{\mathbb{H}}\circ\pi =  \pi\circ S_{\mathbb{G}}, \qquad \epsilon_{\mathbb{G}} = \epsilon_{\mathbb{H}}\circ \pi,
	\]
see Proposition 4.7 of \cite{Born}.   

\begin{prop}
 \label{prop:dual_morphism}
 For any morphism of bornological quantum groups $\pi$ from $\mathbb{H}$ to $\mathbb{G}$, there is a unique dual morphism $\hat\pi$ from $\hat{\mathbb{G}}$ to $\hat{\mathbb{H}}$ determined by
 \[
  (\hat\pi(x),a) = (x,\pi(a))
 \]
 for all $x\in\sA(\hat{\mathbb{H}})$ and $a\in\sA(\mathbb{G})$. 
\end{prop}

\begin{proof}
 The well-definedness of $\hat{\pi}$ is Proposition 8.4 of \cite{Born}.  The compatibility of $\hat{\pi}$ with the involutions follows from duality with $\sA(\mathbb{G})$ and $\sA(\mathbb{H})$.  
\end{proof}

Next we recall the definition of a homomorphism of locally compact quantum groups.  Here, there are several equivalent definitions, as detailed by Meyer, Roy and Woronowicz \cite{MeyRoyWor}, see also \cite[Section 1.3]{Daws}.

\begin{definp} \cite{MeyRoyWor}
\label{bicharacter}
Let $\mathbb{G}$ and $\mathbb{H}$ be locally compact quantum groups. The following objects are in one to one correspondence :
\begin{enumerate}
    \item A \textbf{homomorphism} from  $\mathbb{G}$ and $\mathbb{H}$, that is, a morphism between the universal function algebras
    $$\pi :C_0^u(\GG)\rightarrow M(C_0^u(\HH))  $$
 which intertwines the coproducts.

    \item A \textbf{bicharacter} from  $\mathbb{G}$ and $\mathbb{H}$, that is
    $$V\in M( C_0^r(\HH) \otimes C_0^r(\hat\GG))$$

    satisfying
    \begin{align*}
     (\Delta_\HH\hat\otimes\iota)V &= V_{13} V_{23}, &  
     (\iota \hat\otimes \Delta_{\hat{\GG}})V = V_{13} V_{12}.
    \end{align*}

    \end{enumerate}
\end{definp}
Note that in our conventions, the legs of the bicharacter are flipped with respect to those of \cite{MeyRoyWor, Daws}.

The proof of the equivalence of these definitions relies on lifting bicharacters to the universal algebras.  Explicitly, there is a bicharacter $W^u \in M(C^u_0(\GG) \hat\otimes C^*_u(\GG))$, called the universal multiplicative unitary, which is uniquely characterized by the fact that its image under the regular representations is the usual multicative unitary $W \in M(C^r_0(\GG) \hat\otimes C^*_r(\GG))$.  The universal bicharacter can be defined by $V^u = (\pi\hat\otimes\iota)(W^u)$.

To make the connection with bornological quantum groups, we have the following simple construction.

\begin{prop}
 \label{prop:reduced_bicharacter}
 Let $\pi:\sA(\GG) \to \sM(\sA(\HH))$ be a morphism of bornological quantum groups.  The element
 \[
  (m_\HH\circ\pi \hat\otimes \lambda_\GG)(\sW) \in M(C_0^r(\HH) \hat\otimes C^*_r(\GG))
 \]
 is a bicharacter from $\GG$ to $\HH$ as locally compact quantum groups. 
\end{prop}

\begin{proof}
 
Given a morphism $\pi$ of bornological quantum groups from $\HH$ to $\GG$ as above, we define a \emph{bornological bicharacter}
\[
 \sV = (\pi\hat\otimes\id)\sW \in \sM(\sA(\HH)\btimes\sD(\GG)).
\]

It satisfies the properties
\begin{align}
\label{eq:bicharacter}
    (\Delta_\HH \hat\otimes \id)\sV & = \sV_{13}\sV_{23}, &
    (\id \hat\otimes \Delta_{\hat{\GG}})\sV &= \sV_{13}\sV_{12}, 
\end{align}
thanks to the analogous properties of $\sW$, see Equation \eqref{eq:W_bicharacter}.

Moreover, it is a unitary multiplier in the same sense as $\sW$ from Proposition \ref{prop:sW}.  Therefore, it maps under the regular representations $m_\GG \hat\otimes \lambda_\HH$ to a unitary Hilbert space operator $V$ on $L^2(G)\hat\otimes L^2(H)$ which is a unitary bicharacter in the $C^*$-algebraic sense.
\end{proof}

Combining Proposition \ref{prop:reduced_bicharacter} with Definition-Proposition \ref{bicharacter}, the above proposition yields a Hopf $*$-morphism
\[
  \tilde\pi : C^u_0(\GG) \to M(C^u_0(\HH))
\]
associated to any morphism $\pi$ of bornological quantum groups.
As mentioned above, this morphism is obtained by passing via the universal bicharacter $V^u$, which can be made explicit as follows.  

\begin{lm}
  \label{lem:universal_bicharacter}
  The element
 \[
  V^u = (m^u_\HH\circ\pi \hat\otimes \lambda^u_\GG)(\sW) 
 \]
 is the universal bicharacter associated to the bicharacter $V$ from Proposition \ref{prop:reduced_bicharacter}, where $m^u_\HH$ and $\lambda^u_\GG$ denote universal representations of $\sA(\HH)$ and $\sD(\GG)$, respectively. 
\end{lm}

\begin{proof}

The proof is essentially the same as that of Proposition \ref{prop:reduced_bicharacter}.  The operator $(m^u_\HH\circ\pi \hat\otimes \lambda^u_\GG)(\sW)$ is a densly defined multiplier of $C_0^u(\GG)\hat\otimes C^*_u(\GG)$.  It is unitary on its domain, so extends to a bounded multiplier, and again satisfies the bicharacter properties.  Applying the regular representations, $V^u$ maps to the reduced bicharacter $V \in C_0^r(\HH) \hat\otimes C^*_r(\GG)$ from the previous lemma.  This characterizes the universal bicharacter uniquely, see \cite{MeyRoyWor}.

\end{proof}

We are now in a position to directly compare the bornological and $C^*$-algebraic maps arising from a homomorphism of bornological quantum groups.

\begin{thm}
 \label{thm:C-star-morphism}
 Let $\pi:\sA(\GG) \to \sM(\sA(\HH))$ be a morphism of bornological quantum groups from $\HH$ to $\GG$.  We have a commuting diagram
 \[
  \xymatrix{
  \sA(\GG) \ar[r]^{\pi} \ar[d]_{m^u_\GG} &
  \pi(\sA(\GG)) ~\lefteqn{\subseteq \sM(\sA(\HH))}
  \ar[d]^{m^u_\HH} \\ 
  C^u_0(\GG) \ar[r]^{\tilde\pi} &
  M(C^u_0(\HH))
  }
 \]
 where the vertical arrows are the natural inclusions.  
\end{thm}

\begin{rmk} 
Note that one cannot define the right-hand vertical map directly on the bornological multipliers in $\sM(\sA(\HH))$, since these will generally map to unbounded multipliers of $C^u_0(\HH)$.

The extension of the universal representation $m^u_\HH:\sA(\HH) \to C_0^u(\HH)$ to $\pi(\sA(\GG))$ is made explicit in the proof below.

\end{rmk}

\begin{proof}
It is a consequence of Proposition \ref{prop:W-slices} that any element $a \in \sA(\GG)$ can be written as $a = (\id\hat\otimes \omega)\sW_\GG$ for some $\omega\in\sA(\GG)\subseteq\sD(\GG)^*$.
 Then we have
 \[
  \pi(a) = (\pi\hat\otimes \omega)\sW_\GG = (\id\hat\otimes \omega)\sV,
 \]
 where $\sV = (\pi\hat\otimes \omega)\sW$ is the bornological bicharacter associated to the morphism $\pi$, as above.

 We can then define
 \[
  m^u_\HH(\pi(a)) 
   = ((m^u_\HH\circ\pi) \hat\otimes \omega)(\sW_\GG)
   = (\iota\hat\otimes\tilde{\omega})(V^u),
 \]
 where the second equality uses Lemma \ref{lem:universal_bicharacter} and $\tilde\omega$ denotes the image of $\omega$ in the Fourier algebra $A(\GG)$, see Proposition \ref{prop:DG_in_L1G} and the remarks that follow it.  This map is well-defined because if $\pi(a)=0$ then $(\pi\otimes\omega)\sW_\GG=0$ and so the expression defining $m^u_\HH(\pi(a))$ is zero.
 
 To check that the diagram commutes, we note that the image of $a=(\iota\otimes\omega)\sW_\GG$ under $m^u_\GG$ is $(\iota\hat\otimes\tilde\omega)W_\GG^u$, so that $\tilde{\pi}\circ m^u_\GG(a) = (\iota\hat\otimes\tilde\omega)(V^u)$, as desired.

\end{proof}

\subsection{Closed quantum subgroups in the operator algebra picture}
In this section, we will show that a closed quantum subgroup of a bornological quantum group gives rise to a closed quantum subgroup of the associated locally compact quantum group, in the sense of Vaes. The following definition comes from \cite{Vaes}.

\begin{defin}
	\label{def:Vaes_subgroup}
	Let $\mathbb{G}$ be locally compact quantum groups.  A \emph{closed quantum subgroup of $\mathbb{G}$ in the sense of Vaes} is a locally compact quantum group $\mathbb{H}$ which fits into a commuting diagram
	\[
	\xymatrix{
		C^*_{\mathrm{u}}(\mathbb{H}) \ar[r]^{\hat\pi} \ar[d]_{\lambda_{\mathbb{H}}} &
		M(C^*_{\mathrm{u}}(\mathbb{G})) \ar[d]^{\lambda_{\mathbb{G}}} \\
		\sL(\mathbb{H}) \ar[r]^{\hat\pi} &
		\sL(\mathbb{G})
	}
	\]
	where the top arrow is an essential morphism of Hopf $C^*$-algebras, the bottom arrow is an injective normal unital $*$-homomorphism, and the vertical maps are the regular representations.
\end{defin}

There is another possible definition due to Woronowicz which is weaker that that of Vaes, see \cite[Definition 3.2 and Theorem 3.5]{Daws}.  We will show that closed quantum subgroups in the bornological setting give rise to closed quantum subgroups in the sense of Vaes, and hence also in the sense of Woronowicz.

Throughout this section, $\mathbb{G}$ is a bornological quantum group and $\mathbb{H}$ a closed quantum subgroup with restriction map $\pi_{\mathbb{H}}:\sA(\mathbb{G}) \onto \sA(\mathbb{H})$.


%
Before proving the main theorem, let us begin with some explicit formulas for the dual morphism $\hat{\pi} : \sD(\HH) \to \sM(\sD(\GG))$, as obtained from Proposition \ref{prop:dual_morphism}.

\begin{lm}
\label{lem:dual_to_restriction}
Let $x\in\sD(\HH)$.  The convolution multiplier $\hat{\pi}(x) \in \sM(\sD(\GG))$ is given by
 \begin{align*}
	  \hat{\pi}(x)*u &= \phi_{\mathbb{H}}(S^{-1}(\pi(u_{(1)}))x) u_{(2)}, \\
	  u*\hat{\pi}(x) &= u_{(1)}\, \phi_{\mathbb{H}}(\pi(\delta_{\mathbb{G}} S(u_{(2)}))x)
	    = u_{(1)}\, \phi_{\mathbb{H}} (S^{-1}(x)\pi(u_{(2)})\pi(\delta_{\mathbb{G}}^{-1})\delta_{\mathbb{H}}).
 \end{align*}
 for all $u\in\sD(\GG)$
\end{lm}
 
\begin{proof}
 Let $a\in\sA(\GG)$.  By the definition of $\hat\pi$ we have
 \begin{align*}
 (\phi_{\HH} & (S^{-1}(\pi(u_{(1)}))x) u_{(2)}, a) \\
  &= (\phi_\HH\hat\otimes\phi_\GG)( S^{-1}(\pi(u_{(1)}))x \otimes a u_{(2)}) \\
  &= (\phi_\HH\hat\otimes\phi_\GG)( S^{-1}(\pi(S(a_{(1)})))x \otimes a_{(2)} u) \\
  &= (x \otimes u, \pi(a_{(1)})\otimes a_{(2)}) \\
  &= (\hat\pi(x) \otimes u, \Delta(a)) \\
  &= (\hat\pi(x)*u,a). 
 \end{align*}
This proves the first equation.  The second is similar.
\end{proof} 
 
 Now, a straightforward calculation yields the following fact.

 \begin{prop}
  \label{prop:pi-hat_expectation}
  For any $x,y\in\sD(\mathbb{H})$ and $u\in\sD(\mathbb{G})$ we have
    $\pi(\hat\pi(x)*u*\hat\pi(y)) = x*\pi(u)*y$.
 \end{prop}

\begin{rmk}
  Proposition \ref{prop:pi-hat_expectation} can be interpreted as a conditional expectation property of $\pi$, seen as a linear map $\pi:\sD(\mathbb{G}) \onto \sD(\mathbb{H})$.  Note though that $\pi$ does not necessarily respect the convolution involutions $*$ on $\sD(\mathbb{G})$ and $\sD(\mathbb{H})$, unlike the pointwise involutions $\bar{~}$ on  $\sA(\mathbb{G})$ and $\sA(\mathbb{H})$.  This is an important issue for unitary representation theory, and is addressed in the article \cite{riv}.
\end{rmk}

\begin{thm}
	Let $\mathbb{H}$ be a closed quantum subgroup of a bornological quantum group $\mathbb{G}$.  Then the corresponding locally compact quantum group $\mathbb{H}$ is a closed quantum subgroup of the locally compact quantum group $\mathbb{G}$ in the sense of Vaes (and hence Woronowicz).
\end{thm}

\begin{proof}
 Explicitly, we will show that there is a commuting diagram
 \[
 \xymatrix{
  \sD(\mathbb{H}) \ar[r]^{\hat\pi} \ar[d]_{\lambda^u_{\mathbb{H}}} &
  \hat\pi(\sD(\HH))
  ~\lefteqn{\subseteq \sM(\sD(\mathbb{G}))}
  \ar[d]^{\lambda^u_\mathbb{G}} \\
  C^*_{\mathrm{u}}(\mathbb{H}) \ar[r]^{\hat\pi} \ar[d]
  &
  M(C^*_{\mathrm{u}}(\mathbb{G})) \ar[d]
  \\
  \sL(\mathbb{H}) \ar[r]^{\hat\pi} &
  \sL(\mathbb{G}).
 }
 \]

 The top square is the dual of the commuting square from Theorem \ref{thm:C-star-morphism} (we are suppressing the tilde from the horizontal $C^*$-algebra map $\tilde{\hat{\pi}}$ to simplify the notation).
 For the von Neumann morphism, for any $x\in\sD(\mathbb{H})$ and $b\in\sA(\mathbb{G})$ we have
 $
  (\hat\pi(x) , b) = (x, \pi(b)),
 $
 and it follows that $\hat\pi$ is ultraweakly continuous, so can be extended to a normal unital $*$-homomorphism.  The outer rectangle is commutative and thus, by the density of $\sD(\HH)$ in $C^*_u(\HH)$, the bottom square is also commutative.
 The crucial point is to prove that the von Neumann algebra map is injective.

 From Proposition \ref{prop:DG_in_L1G}, $\sA(\mathbb{H}) = \sD(\hat{\mathbb{H}})$ embeds as a dense subspace of the predual $\sL(\mathbb{H})_*$.  Explicitly, we identify $a\in\sA(\mathbb{H})$ with the functional $\tilde{a}\in\sL(\mathbb{H})_*$ where
 \[
  \tilde{a}(x) 
   = \hat\phi_{\mathbb{H}}(\sF(x)\sF(a)) 
   = \epsilon_{\mathbb{H}}(x*a)
 \]
 for $a\in\sA(\mathbb{H})$, $x\in\sD(\mathbb{H})$.  Choose $b\in\sA(\mathbb{G})$ with $\pi(b)=a$.  
 Using Proposition \ref{prop:pi-hat_expectation}, we have
 \[
  \tilde{a}(x) = \epsilon_{\mathbb{H}}(x*\pi(b)) 
   = \epsilon_{\mathbb{H}}(\pi(\hat\pi(x)*b))
   = \epsilon_{\mathbb{G}}(\hat\pi(x)*b)
   = \hat{b}(\hat\pi(x)),
 \]
 for all $x\in\sD(\mathbb{H})$, and hence $\tilde{a}(x) = \hat{b}(\hat\pi(x))$ for all $x\in\sL(\mathbb{G})$ by ultraweak continuity.	
 Therefore, if $x\in\sL(\mathbb{H})$ is in the kernel of $\hat\pi$ then $x$ is annihilated by all of the functionals $\tilde{a}$ with $a\in\sA(\mathbb{H})$.  These are dense in the $\sL(\mathbb{H})_*$ so $x=0$.
\end{proof}

\bibliographystyle{alpha}
\bibliography{ref.bib}
\end{document}